\documentclass[12pt]{amsart}
\usepackage{amssymb,amsthm,amsmath,amsfonts,amscd}
\usepackage{thmtools, thm-restate}
\numberwithin{equation}{section}
\numberwithin{figure}{section}
\usepackage{array}
\usepackage{calc,graphicx, mathrsfs}
\usepackage{microtype}
\usepackage{tikz}
\usetikzlibrary{cd}
\usepackage{bm,enumitem}
\usepackage[width=.75\textwidth]{caption}
\usepackage{subcaption,xspace}

\usepackage[margin=1in]{geometry}

\DeclareFontFamily{OT1}{pzc}{}
\DeclareFontShape{OT1}{pzc}{m}{it}{<-> s * [1.10] pzcmi7t}{}
\DeclareMathAlphabet{\mathpzc}{OT1}{pzc}{m}{it}

\title[Spectral sequences and link homology]{The Ozsv\'{a}th-Szab\'{o} spectral sequence and combinatorial link homology}
\author{Adam Saltz}
\email{adam.saltz@uga.edu}

\newcommand{\stack}[2]{\vtop{\hbox{\strut{#1}} \hbox{\strut{#2}}}}

\newcommand{\bF}{\mathbb{F}}

\newcommand{\bz}{\mathbb{Z}}
\newcommand{\bq}{\mathbb{Q}}
\newcommand{\br}{\mathbb{R}} 

\newcommand{\al}{\alpha}
\newcommand{\be}{\beta}
\newcommand{\ga}{\gamma}

\newcommand{\bal}{\bm{\al}}
\newcommand{\bbe}{\bm{\be}}

\newcommand{\bet}{\bm{\eta}}

\newcommand{\boldA}{\mathbf{A}}
\newcommand{\boldB}{\mathbf{B}}

\newcommand{\diagram}{\mathpzc{D}}
\newcommand{\bouquet}{\mathpzc{Bo}}
\newcommand{\branched}{\mathpzc{Br}}
\newcommand{\fraks}{\mathfrak{s}}
\newcommand{\CF}{\widehat{CF}}
\newcommand{\HF}{\widehat{HF}}
\newcommand{\HFK}{\widehat{HFK}}

\DeclareMathOperator{\liftA}{\tilde{A}}
\DeclareMathOperator{\alex}{\mathbf{A}}

\DeclareMathOperator{\rank}{rank}

\DeclareMathOperator{\Sym}{Sym}

\DeclareMathOperator{\Spin}{Spin}
\newcommand{\Spinc}{\Spin^c}
\DeclareMathOperator{\gr}{\widetilde{gr}}

\newcommand{\Ozsvath}{Ozsv{\'a}th\xspace}
\newcommand{\Szabo}{Szab{\'o}\xspace}

\DeclareMathOperator{\Span}{span}

\DeclareMathOperator{\OS}{OS}
\DeclareMathOperator{\BR}{BR}
\DeclareMathOperator{\Sz}{Sz}
\DeclareMathOperator{\CSz}{CSz}
\DeclareMathOperator{\ESz}{ESz}

\DeclareMathOperator{\basez}{\mathbf{z}}
\DeclareMathOperator{\basew}{\mathbf{w}}

\newcommand{\cmap}{\mathcal{F}}
\newcommand{\config}{\mathcal{C}}
\newcommand{\decos}{\mathbf{t}}

\DeclareMathOperator{\im}{Im}

\DeclareMathOperator{\Id}{Id}

\DeclareMathOperator{\CKh}{CKh}
\DeclareMathOperator{\Kh}{Kh}
\DeclareMathOperator{\AKh}{AKh}

\def\co{\colon\thinspace}

\newtheorem{Thm}{Theorem}
\newtheorem*{Thm*}{Theorem}
\newtheorem{Prop}[Thm]{Proposition}
\newtheorem{Lemma}[Thm]{Lemma}
\newtheorem{Cor}[Thm]{Corollary}
\newtheorem{Conj}{Conjecture}

\theoremstyle{definition}
\newtheorem{Def}[Thm]{Definition}

\theoremstyle{remark}
\newtheorem{Rem}{Remark}

\begin{document}
\bibliographystyle{plain}
\begin{abstract} 
The Khovanov homology of a link in $S^3$ and the Heegaard Floer homology of its branched double cover are related through a spectral sequence constructed by \Ozsvath and \Szabo.  This spectral sequence has topological applications but is difficult to compute.  We build an isomorphic spectral sequence whose underlying filtered complex is as simple as possible: it has the same rank as the Khovanov chain group.  We show that this spectral sequence is not isomorphic to \Szabo's combinatorial spectral sequence, which Seed and \Szabo conjectured to be equivalent to \Ozsvath-\Szabo's.  The discrepancy leads us to define a variation of \Szabo's theory for links embedded in a thickened annulus.  We conclude with a refinement of Seed and \Szabo's conjecture for the new theory.
\end{abstract}
\maketitle

\section{Introduction}

Khovanov homology and Heegaard Floer homology have each been used to prove and reprove remarkable theorems in low-dimensional topology.  Among the most amazing of these theorems is the relationship between the two invariants.

\begin{Thm*}[\Ozsvath and \Szabo, \cite{Ozsvath2005b}]
  Let $L \subset S^3$ be a link with mirror $m(L)$.  Let $S^3(L)$ be the double cover of $S^3$ branched along $L$.  There is a spectral sequence from the Khovanov homology $\Kh(m(L))$ to the Heegaard Floer homology $\HF(S^3(L)) \otimes \HF(S^1 \times S^2)$.
\end{Thm*}

Beyond the theoretical interest of connecting representation theory and symplectic topology, this theorem has applications to the study of transverse links and contact structures \cite{Baldwin2010} and twist numbers of braids \cite{Hedden2015}.   Baldwin \cite{Baldwin2011} has shown that each page of this spectral sequence is a link invariant.  In other words, the spectral sequence defines a sequence of link invariants which somehow interpolate between a link's Khovanov homology and the Heegaard Floer homology of its branched double cover, and this interpolation carries contact-theoretic information.

Despite all this interest, computing the pages of the spectral sequence remains a challenge.  For a link diagram $\diagram$, write $\OS(\diagram)$ for the spectral sequence from the theorem.  Computing $\OS(\diagram)$ directly involves counting holomorphic polygons in a high-dimensional symplectic manifold.  On the other hand, Khovanov homology is relatively easy to compute.  \Szabo \cite{Szabo2013} defined a combinatorial spectral sequence $\ESz(\diagram)$ in the style of Khovanov homology and conjectured that it is equivalent to $\OS(\diagram)$.  

\begin{restatable}{Conj}{seedszabo}\label{conj:original}\cite{Seed2011, Szabo2013}
  Let $\diagram$ be a diagram for a link $L \subset S^3$.
  \begin{itemize}
    \item (stronger) $\OS(\diagram) \cong \ESz(\diagram)$.
    \item (weaker) $\Sz(m(L)) \cong \HF(L) \otimes \HF(S^1 \times S^2)$.
  \end{itemize}
\end{restatable}

Seed \cite{Seed2011} wrote software to compute \Szabo's spectral sequence and confirmed the conjecture for knots with at most twelve crossings.

This project began with an attempt to prove Conjecture~\ref{conj:original} by using simpler Heegaard diagrams than those of \cite{Ozsvath2005b}.  The zeroth page of $\OS(\diagram)$ is a direct sum of Heegaard Floer chain groups whose rank may be larger than the Khovanov chain group $\CKh(m(\diagram))$.  The first page of $\OS(\diagram)$ is obtained by taking the homology of each individual Heegaard Floer group, and this page is isomorphic to $\CKh(m(\diagram))$.  Meanwhile, \Szabo's spectral sequence starts from the Khovanov chain group.  So a first step towards proving Conjecture~\ref{conj:original} is build $\OS(\diagram)$ from Heegaard diagrams which are more obviously connected to Khovanov homology.  These diagrams, which we called \emph{branched}, first appeared in somewhat different form in~\cite{Grigsby2011}.  The Heegaard Floer chain groups of these diagrams have minimal rank, i.e.\ the differential vanishes.

\begin{restatable}{Thm}{mainthm}\label{thm:main}
  Let $\diagram$ be a diagram of a link $L \subset S^3$.  There is a spectral sequence $\BR(\diagram)$, built from branched diagrams, so that $\BR^0(\diagram) \cong \BR^1(\diagram) \cong \CKh(\diagram)$.  This spectral sequence is isomorphic to the \Ozsvath-\Szabo spectral sequence after the zeroth page.  In particular, $\BR^2(\diagram) \cong \Kh(m(L))$ and $\BR^\infty(\diagram) \cong \HF(S^3(L)) \otimes \HF(S^1 \times S^2)$.
\end{restatable}

We had hoped that the simplicity of branched diagrams would allow us to count holomorphic polygons and identify the maps in $\BR(\diagram)$ with those in $\ESz(\diagram)$.  Combined with the theorem, this would prove Conjecture~\ref{conj:original}.  In fact, these maps are different.  We propose that the discrepancy arises from the asymmetry with respect to basepoints between \Szabo (and Khovanov's) theory and Heegaard Floer homology.  Given planar isotopic diagrams $\diagram$ and $\diagram'$, the construction underlying Theorem \ref{thm:main} provides an obvious isomorphism $\BR(\diagram) \cong \BR(\diagram')$ by a based isotopy of Heegaard multi-diagrams.  But \Szabo's theory applies to \emph{spherical} link diagrams.  An isotopy from $\diagram$ to $\diagram'$ which crosses the point at infinity on $S^2$ induces an isotopy of Heegaard diagrams which crosses a basepoint.  The rank of $\BR(\diagram)$ does not change, but the gradings on individual generators (and therefore the correspondence with Khovanov homology) does.

We propose that this discrepancy can be understood by studying embeddings of knots in thickened annuli, or equivalently projections of links to the complement of two points in $S^2$.  There is a well-known \emph{annular grading} $k$ on the Khovanov homology of annular links, see~\cite{Roberts2013}.  We construct a doubly-pointed \Szabo homology theory which takes $k$ into account.  On the Heegaard Floer side, the two basepoints describe a link in each branched double cover.  This introduces a grading $\alex$, the \emph{Alexander grading}, on all the Heegaard Floer groups.  We make the following conjecture precise in the last section.

\begin{Conj}
  Let $\diagram$ be a spherical link diagram with basepoints $z, w \in S^2 \setminus \diagram$.  There is a spectral sequence $\ESz_{z,w}$ from the doubly-pointed \Szabo chain complex to $\HF(S^3(L)) \otimes \HF(S^1 \times S^2)$.  $\ESz_{z,w}$ is isomorphic to $\BR$ and therefore to $\OS$.  This isomorphism identifies the grading $q - 2k$ on $\CKh(\diagram)$ with $2(\gr - 2\alex)$ on $\HF(S^3(L)) \otimes \HF(S^1 \times S^2)$.
\end{Conj}

$q$ is the quantum grading on Khovanov homology, $\gr$ is the Maslov grading, and $\alex$ is the Alexander grading induced by the pre-image in $S^3(L)$ of the unknot specified by $z$ and $w$.  This implies that there is a consistent grading on each page of the spectral sequence.  For $\OS$ this was originally conjectured by Baldwin (without the Alexander correction) in~\cite{Baldwin2011} building on Greene~\cite{Greene2013}.  This conjecture is consistent with all holomorphic polygon counts we are able to complete, and at the level of homology it is consistent with Seed's computations.  It implies Conjecture~\ref{conj:original}.  It suggests that $\ESz_{z,w}$ is more closely geometrically connected to Heegaard Floer homology than $\ESz$ and therefore is a more natural object of study. 

Experts will recognize $\gr - \alex$ as the $\delta$-grading on knot Floer homology.  Using the recipe from \cite{Sarkar2016} or \cite{Hunt2015}, one could define $\ESz_{z,w}$ over the polynomial ring $\bz/2\bz[K,K^{-1}]$ where the variable $K$ accounts for the $k$-grading.  Many variants of knot Floer homology are naturally defined over the ring $\bz/2\bz[U,U^{-1}]$.  We speculate that there is some connection between the $U$- and $K$-actions.

All of the constructions of paper use coefficients in $\bz/2\bz$, which we denote by $\bF$.  \Szabo homology with $\bz$-coefficients was introduced in~\cite{Beier2012}.  (Note that the resulting spectral sequence is from \emph{odd} Khovanov homology.)  Heegaard Floer homology may be defined with $\bz$-coefficients, but we have not carefully checked that the proof of Theorem~\ref{thm:main} carries over.

We begin by reviewing the construction of the \Ozsvath-\Szabo spectral sequence in preparation for Section~\ref{sec:handleslides}, in which we prove the main theorem.  We construct branched diagrams and prove some of their basic properties in Section~\ref{sec:brancheddiagrams}.  In Section~\ref{sec:szabo} we recall the definition of \Szabo homology and construct doubly-pointed \Szabo homology.  We assume some familiarity with Khovanov homology and Heegaard Floer homology throughout.

\subsection*{Acknowledgments} This project was suggested to me by John Baldwin when I was his student.  Not surprisingly, it owes a lot to his help.  I am also thankful to Eli Grigsby and Josh Greene for helpful conversations and suggestions.  

\section{The \Ozsvath-\Szabo spectral sequence}\label{sec:ozssz}

The construction of $\OS(\diagram)$ first appears in \cite{Ozsvath2005b}.  For more details, see~\cite{Roberts2008}.  Baldwin has shown that this construction does not depend on choices of analytic data~\cite{Baldwin2011}, so we will ignore them.  Our presentation will use Heegaard diagrams with two basepoints as described by Manolescu and \Ozsvath in~\cite{Manolescu2017} but at a much lower level of generality.  (Multi-pointed Heegaard diagrams for three-manifolds also appear in \cite{Ozsvath2008}.)

\begin{Def}
  Let $\Sigma$ be an oriented, genus $g$ surface.  Let $\bal = \{\al_1, \ldots, \al_{g + 1}\}$ and $\bet = \{\eta_1, \ldots, \eta_{g+1}\}$ each be sets of simple, disjoint, closed curves on $\Sigma$ which span $g$-dimensional lattices $H_1(\Sigma; \bz)$ (i.e.\ they each specify a handlebody bounded by $\Sigma$).  Assume that all the intersections between $\al$ and $\eta$ curves are transverse.  Let $\basew = \{w_1, w_2\}$ be a pair of basepoints which lie in different connected components of the complements of $\bal$ and $\bet$ in $\Sigma$.  The data $(\Sigma, \bal, \bet, \basew)$  is a \emph{two-pointed, balanced Heegaard diagram}.
\end{Def}

A standard Heegaard diagram specifies a cellular decomposition of a three-manifold with one $0$- and $3$-cell and $g$ $1$- and $2$-cells.  Likewise, a two-pointed, balanced Heegaard diagram specifies a cellular decomposition of a three-manifold with two $0$- and $3$-cells.  Alternatively, one may view two-pointed Heegaard diagram as a byproduct of a self-indexing Morse function on a three-manifold with two index zero and two index three critical points.

\begin{Rem}
Do not confuse these two-pointed Heegaard diagrams with those used in knot Floer homology.  In the classic knot Floer setting, a genus $g$ diagram representing a knot has two sets of $g$ curves and two basepoints.  Our diagrams have two sets of $g+1$ and two basepoints.
\end{Rem}

Let $Y$ be a closed three-manifold.  Let $K \subset Y$ be a link with $k$ components.  Choose some point $z \in Y \setminus K$.   Connect each component of $K$ to $z$ through pairwise disjoint arcs.   The union of $K$ with such a collection of arcs is called a \emph{bouquet} for $K$.

\begin{Def} Let $\Gamma$ be a bouquet for $K$.  A balanced, two-pointed Heegaard diagram $(\Sigma, \{\al_i\}_{i=1}^{g+1}, \{\eta_i\}_{i=1}^{g+1}, \basew)$ presenting $Y$ is \emph{subordinate to the bouquet $\Gamma$} if it satisfies the following conditions.
\begin{itemize}
\item For some $k$, the diagram $(\Sigma, \bal, \{\eta_{k+1}, \ldots, \eta_{g+1}\}, \basew)$ presents $Y \setminus N(\Gamma)$ where $N(\Gamma)$ is a normal neighborhood of $\Gamma$.
\item Surger $\eta_{k+1}, \ldots, \eta_{g+1}$ from $\Sigma$.  Each remaining $\eta_i$ lies on a punctured torus $\partial N_i \subset N(\Gamma)$ which surrounds the component $K_i$ for $1 \leq i \leq k$.
\item For $1 \leq i \leq k$, the curve $\eta_i$ is a meridian of $K_i$.
\end{itemize}
\end{Def}

Loosely, the first $k$ of the $\bet$ curves describe a normal neighborhood of $\Gamma$ and the remaining $\bet$ curves fill out the rest of $Y$.  Such a diagram exists for every pair $(Y,K)$, and the constructions in this section do not depend on a choice of bouquet~\cite{Roberts2013}.

Designate $\eta_1, \ldots, \eta_k$ as $\infty$-framed curves along the tori $N_i$.  Let $\gamma_i$ and $\delta_i$ be curves on $N_i$ with framings $0$ and $1$, respectively.  Without loss of generality, we may arrange that $|\be_i \cap \delta_i| = |\delta_i \cap \gamma_i| = |\gamma_i \cap \be_i| = -1$ and that each of these curves is disjoint from $\basew$.  For $I \in \{0,1,\infty\}^k$, define the set of curves $\eta(I)$ by
\[
	\eta_j(I) = 
	\begin{cases} \eta_i & \text{if } j > k \text{ or } I_j = \infty \\
	\ga_i & \text{if } I_j = 0 \\
	\delta_i & \text{if } I_j = 1 \end{cases}
\]
The diagram $(\Sigma, \bal, \bet(I),\basew)$ is a balanced, two-pointed Heegaard diagram for the result of $I$-framed surgery along $K$.

In all, we have constructed a family of Heegaard diagrams which realize all $0$- and $1$-surgeries along components of $K$.  For the rest of the section, $K$ will play an auxiliary role to the link $L$ whose Khovanov- and Floer-theoretic invariants will be related.

\begin{figure}[h]
\centering
 \includegraphics[width=.66\textwidth]{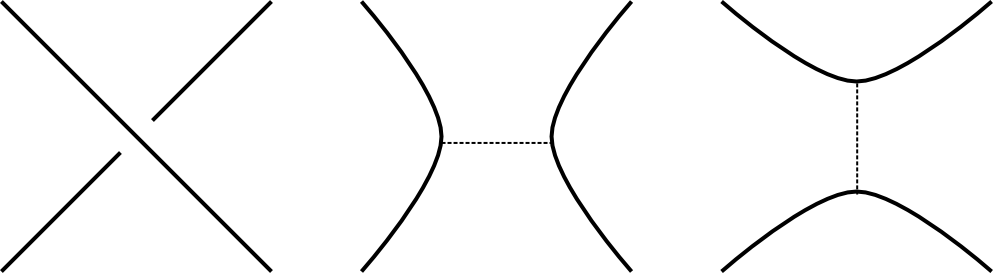}
 \caption{A crossing, its 0-resolution, and its 1-resolution with surgery arcs.  Note that this is the opposite of the usual convention in Khovanov homology.  This is why the \Ozsvath-\Szabo spectral sequence involves either the mirror of the link or the opposite of its branched double cover.}
\label{fig:resolutions}
 \end{figure}

Let $\diagram$ be a diagram for a link $L \subset S^3$ with $c$ crossings.  Each crossing of $\diagram$ may be resolved in one of the two ways shown in Figure~\ref{fig:resolutions}.  After fixing some ordering on the crossings, the complete resolutions of $\diagram$ are indexed by elements of $\{0,1\}^c$, the \emph{cube of resolutions}.   The resolutions of a single crossing are ordered by $0 < 1$ and the cube of resolutions is partially ordered by the product order.  Call the resolution $I'$ a \emph{successor} of $I$ and write $I \prec I'$ if $I < I'$ and the two differ in exactly one entry.  A \emph{path of resolutions} $\mathbf{I}$ is a chain of successors, i.e.\ $\mathbf{I} = \{I_1, I_2, \ldots, I_n \mid I_{i+1} \prec I_i\}$.  A path of length one is a single resolution.  Write $\diagram(I)$ for the diagram resulting from resolving $\diagram$ according to $I$.

Changes of resolution can be realized as surgeries on branched double covers.  Write $I^0$ for the resolution $(0, \ldots, 0)$.  In the diagram $\diagram(I^0)$ draw small auxiliary arcs connecting the two strands of each resolution as in Figure~\ref{fig:resolutions}.  One may change any $0$-resolution to a $1$-resolution by cutting out a small neighborhood of an auxiliary arc and gluing it back in, rotated $90$ degrees.  This lifts to $S^3(\diagram(I^0))$ as surgery: at each crossing the auxiliary arc lifts to a knot on which $0$-surgery produces $Y(\diagram(0, \ldots, 1, \ldots, 0))$.

Let $K$ be the link of these auxiliary knots in $S^3(D(I^0))$ and fix a $K$-bouquet diagram $(\Sigma, \bal, \bet(I^0), \basew)$.  For any resolution $I$ of $\diagram$ we constructed a bouquet diagram $(\Sigma, \bal, \bet(I), \basew)$ for $S^3(D(I))$.  From a path of resolutions $\mathbf{I} = \{I_1, \ldots, I_n\}$ one can build the \emph{bouquet multi-diagram}
\[
 \bouquet(\mathbf{I}) = (\Sigma, \bal, \bet(I_1), \bet(I_2), \ldots, \bet(I_n), \basew).
\]
In general, $\bet(I_j)$ and $\bet(I_{j+1})$ will have curves in common, which is unsuitable for Heegaard Floer homology.   We implicitly perturb identical curves so that they intersect twice pairwise and so that the resulting multi-diagram is admissible. We call two such curves \emph{parallel}.  Roberts showed that these perturbations can be done in a systematic way~\cite{Roberts2008}.  Note that $(\Sigma, \bet(I_i), \bet(I_{i+1}))$ presents a nearly-standard (two-pointed) diagram for a connected sum of $S^1 \times S^2$s; there are two non-parallel curves, and they meet each other once and miss every other curve, so they form an $S^3$ factor.  

As usual, $\CF(\Sigma, \bal, \bet, \basew)$ is generated over $\bf$ by $(g+1)$-tuples $(c_1, \ldots, c_{g+1})$ of intersection points between $\bal$ and $\bet$ curves so that if $c_i \in \al_j \cap \eta_k$ then $c_{i'} \notin \al_j \cup \eta_k$ for any $i' \neq i$.  We write $\CF(\Sigma, \bal, \bet(I), \basew) = \CF(\bal,\bet(I))$ when $\Sigma$ and $\basew$ are understood.  The differential $d$ on two-pointed $\CF(\bal,\bet(I))$ is essentially the same the single pointed theory's:
\[
d(x) = \sum_{y \in T_{\bal} \cap T_{\bet(I)}} \sum_{\substack{\phi \in \pi_2(x, y) \\ \mu(\phi) = 1,  n_{\basew}(\phi) = 0}} \| \widehat{\mathcal{M}}(\phi)\| y
\]
where
\begin{itemize}
  \item $T_{\bet(I)}$ is the torus $\eta_1(I) \times \cdots \eta_{g+1}(I) \subset \Sym^{g+1}(\Sigma)$ and likewise for $T_{\bal}$.
  \item $\pi_2(x, y)$ is the set of homotopy classes of Whitney bigons with edges along $\bal$ and $\bet$ and corners at $x$ and $y$.
  \item $\mu$ is the Maslov index.
  \item $n_{\basew} \co \pi_2(x, y) \to \bz^2$ records intersection numbers of $\pi \in \pi_2$ with $\{z_1\} \times \Sym^{g}(\Sigma)$ and $\{z_2\} \times \Sym^g(\Sigma)$.
  \item $\mathpzc{M}(\phi)$ is the moduli space of holomorphic representatives of $\phi$.
\end{itemize}.
Under the condition $\mu(\phi) = 1$, the space $\mathpzc{M}(\phi)$ is one-dimensional with a free $\br$-action.  Its quotient $\widehat{\mathpzc{M}}(\phi) = \mathpzc{M}(\phi)/\br$ is zero-dimensional and compact and therefore finite with cardinality $\|\widehat{\mathpzc{M}}(\phi)\|$.  The exact construction depends on analytic data which does not affect our discussion.

\begin{Rem}\label{rem:multitosingly}
There is a recipe for transforming a balanced, two-pointed Heegaard diagram for $Y$ into a singly-pointed Heegaard diagram for $Y \# S^1 \times S^2$.  Perform surgery along the two basepoints, i.e.\ remove a small disk around each basepoint and attach a cylinder along the boundaries.  Place a new basepoint on the cylinder.  The resulting diagram presents $Y \# S^1 \times S^2$.  Write $\mathcal{H}$ for the original diagram and $\mathcal{H'}$ for the new one.  It is clear that $\CF(\mathcal{H})$ and $\CF(\mathcal{H}')$ have the same set of generators.  With the basepoint on the cylinder, they have the same differential as well.  So $\HF(\mathcal{H}) = \HF(\mathcal{H}') = \HF(Y \# S^1 \times S^2)$.  We conclude from the K\"{u}nneth formula for Heegaard Floer homology that $\HF(\mathcal{H}) \cong \HF(Y) \otimes \HF(S^1 \times S^2)$.
\end{Rem}

For each path $\mathbf{I}$ of length greater than one define
\begin{align*}
 f_\mathbf{I}: \widehat{CF}(\bal, \bet(I_1)) \otimes \widehat{CF}(\bet(I_1), \bet(I_2)) \otimes \cdots &\otimes \widehat{CF}(\bet(I_{n-1}), \bet(I_n)) \\
 &\to \widehat{CF}(\bal, \bet(I_n))
\end{align*}
by
\[
	f_\mathbf{I}(x_1 \otimes \cdots \otimes x_n) = \sum_{y \in T_{\bal} \cap T_{\bet(I_n)}} \sum_{\substack{\phi \in \pi_2(x_1, \ldots, x_n, y) \\ \mu(\phi) = 0, \, n_{\basew}(\phi) = 0}} |\mathpzc{M}(\phi)| \, y.
\]

The notation is the same as above, with $\pi_2(x_1, \ldots, x_n, y)$ as the set of Whitney $n$-gons with corners at $x_1, \ldots, x_n, y$.

$(\Sigma, \bet(I_i), \bet(I_{i+1}))$ is a standard diagram for a connected sum of $S^1 \times S^2$s and therefore the group $\widehat{CF}(\bet(I_i), \bet(I_{i+1}))$ has a distinguished generator (and cycle) $\Theta_i$, see~\cite{Ozsvath2004b}.  For each path $\mathbf{I}$ of length greater than one there is a map $d_{\mathbf{I}}: \widehat{CF}(\bal, \bet(I_1)) \to \widehat{CF}(\bal,\bet(I_n))$  defined by
\[
	d_{\mathbf{I}}(x) = f_{\mathbf{I}}(x \otimes \Theta_1 \otimes \cdots \otimes \Theta_{n-1}).
\]
For paths of length one define $d_\mathbf{I}$ to be the usual differential on $\widehat{CF}$.  For two resolutions $I,I'$ (possibly equal), write $\mathpzc{P}(I,I')$ for the set of paths from $I$ to $I'$ and define
\[
	d_{I,I'} = \sum_{\mathbf{I} \in \mathpzc{P}(I,I')} d_{\mathbf{I}}.
\]
Define
\[
	X = \bigoplus_{I \in \{0,1\}^k} \CF(\Sigma, \bal, \bet(I), z)
\]
and define $D: X \to X$ by
\[
	D = \sum d_{I, I'}.
\]

\begin{Thm*}
The pair $(X,D)$ is a complex.  Its homology is isomorphic to $\HF(S^3(L) \# S^1 \times S^2) \cong \HF(S^3(L)) \otimes \HF(S^1 \times S^2)$.
\end{Thm*}
\begin{proof}
  For diagrams with a single basepoint, this is the main theorem of \cite{Ozsvath2005b}.  The case with two basepoints follows from Theorems 11.8 and 11.9 of~\cite{Manolescu2017} which enormously expands on~\cite{Ozsvath2005b}.  We present a more modest extension of the argument.

  Following Remark~\ref{rem:multitosingly}, we may replace each doubly-pointed Heegaard diagram with a singly-pointed diagram.  The process transforms a doubly-pointed bouquet diagram $\mathcal{H}$ for $K \subset Y$ into a singly-pointed bouquet diagram $\mathcal{H}'$ for $K' \subset Y \# (S^1 \times S^2)$ where $K'$ is the image of $K$ under an embedding $Y \setminus B^3 \to Y \# (S^1 \times S^2)$ which respects the connected sum.  After applying a sequence of pointed Reidemeister-Singer moves, this diagram is a connected sum of diagrams for $Y$ and $S^1 \times S^2$ so that the basepoint lies in the connected sum region.  (This is Haken's lemma,~\cite{Hempel2004}, for pointed diagrams.)    Every pointed Reidemeister-Singer move on $\mathcal{H}'$ corresponds to a doubly-pointed Reidemeister-Singer move on $\mathcal{H}$.  Thus the sequence of pointed moves gives rise to a sequence of doubly-pointed moves.

  Now apply this process to all of the diagrams underlying $X$ to get an identical complex $X'$ (same generators, same differential).  By~\cite{Ozsvath2005b} and the K\:{u}nneth formula, the homology of $X'$ is isomorphic to $\HF(S^3(m(L))) \otimes \HF(S^1 \times S^2)$.
\end{proof}

The complex $(X, D)$ is filtered by the partial ordering on $\{0,1\}^k$ and so its homology can be computed via a spectral sequence as in \cite{McCleary2000}.  This is the \Ozsvath-\Szabo spectral sequence.

\begin{Thm*}[\Ozsvath, \Szabo] Let $\OS^j(\diagram)$ be the spectral sequence induced by the order filtration on $X$.
\begin{itemize} 
  \item $\OS^0(\diagram) = X$ as a group.  The differential $d_0$ is the sum of the internal differentials on each Heegaard Floer chain group.  Each of these groups is equal to $\widehat{CF}(\#^m(S^2\times S^1))$ for some $m \geq 0$.
  \item $\OS^1(\diagram) \cong \CKh(m(\diagram))$, the Khovanov chain group of the mirror of $\diagram$.
  \item $\OS^2(\diagram) \cong \Kh(m(L))$, the Khovanov homology of the mirror of $L$.
  \item $\OS^\infty(\diagram) \cong \widehat{HF}(S^3(L)) \otimes \HF(S^1 \times S^2)$.   
\end{itemize}
\end{Thm*}

In any such spectral sequence the differential on the zeroth page is the part of $D$ which preserves the filtration.  This and the standard computation of $\HF(S^1 \times S^2)$ explains the first two bullets.  To prove the third and fourth facts, one must understand the moduli spaces of holomorphic polygons which contribute to $D$.  Bouquet diagrams are well-adapted to this situation because $\bet(I_i)$ and $\bet(I_{i+1})$ differ only on some punctured torus.

\section{Branched diagrams}\label{sec:brancheddiagrams}

In this section we introduce special diagrams for the branched double cover of a resolution of a link.  They first appeared in a different guise in~\cite{Grigsby2011}.  Branched diagrams are less suited than bouquet diagrams for counting holomorphic polygons, but they have a much clearer connection to Khovanov homology:  $\CF(\branched(\diagram(I))) \cong \CKh(\diagram(I))$ with vanishing Heegaard Floer differential.  In this sense, branched diagrams have minimal complexity among those which realize the \Ozsvath-\Szabo spectral sequence.

\begin{figure}[h]
\centering
\includegraphics[width=.66\textwidth]{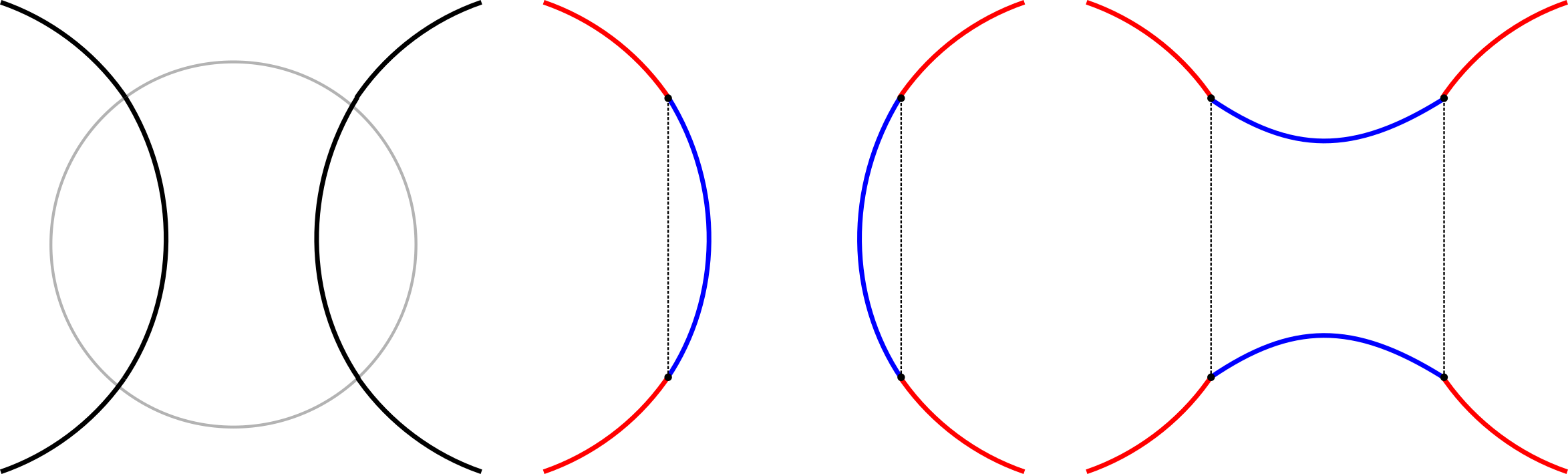}
\caption{On the left, the 0-resolution of a crossing and a small circle.  In the middle, one side of the corresponding local picture on a branched diagram.  (The other side is identical.)  On the right, the same for the 1-resolution.}
\label{fig:localbranched}
\end{figure}

Let $L \subset S^3$ be a link and $\diagram$ a spherical diagram for $L$ with $c$ crossings.  Draw a small circle around each crossing so that it contains exactly two arcs of $L$.  Let $I^0$ be the all-zeros resolution of $\diagram$.  There is a diagram $\diagram(I^0)$ for this resolution which differs from $\diagram$ on in the small circles.  On $\diagram(I^0)$, color the arcs inside the circles blue, color the arcs outside the circles red, and add dotted arcs as shown in Figure~\ref{fig:localbranched}.  For a component with no crossings, draw a dotted arc to divide it into two pieces and color one side red and the other blue.  Now make a copy of this diagram.  Properly interpreted, this pair of pictures is a Heegaard diagram.  The dotted arcs are branch cuts connecting the two spheres to form a surface $V$.  The red and blue arcs on each side are connected through the branch cuts to form sets of red and blue circles which we denote by $\boldA$ and $\boldB(I^0)$, respectively.

\begin{figure}[h]
\centering
\def\svgwidth{.66\columnwidth}
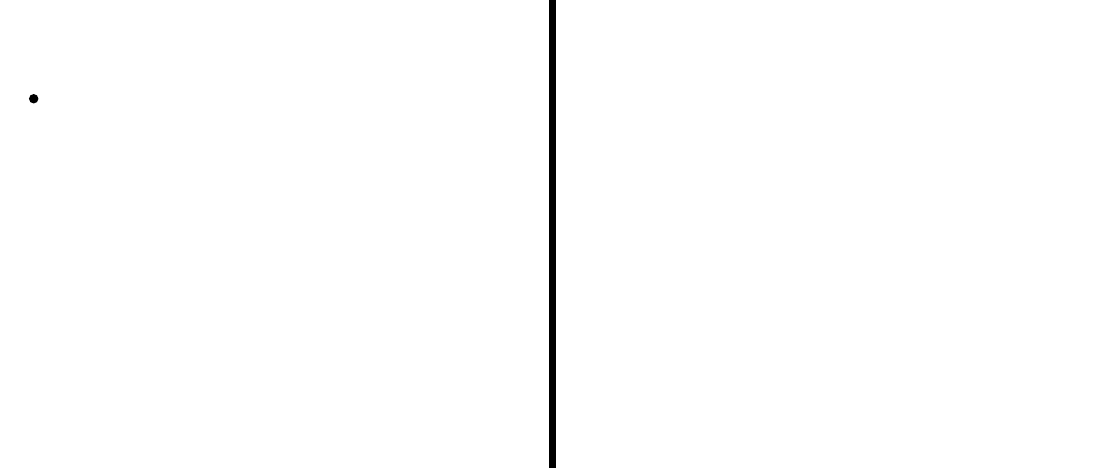
\caption{A branched diagram from the $(1,0,0)$ resolution of a trefoil with the crossings ordered clockwise from the top left.  The bar in the middle indicates that the two diagrams lie on different spheres, connected through the branch cuts.}
\label{fig:trefoilbranched}
\end{figure}

\begin{Prop} The Heegaard diagram $(V, \boldA, \boldB(I^0))$ presents $S^3(\diagram(I^0))$.\end{Prop}

\begin{proof} Project the diagram $\diagram(I^0)$, along with the small circles around the crossings, onto a $2$-sphere $S\subset S^3$.  There are $4k$ points at which the diagram meets the small circles.  Leaving those fixed, gently lift the blue curves off of the sphere and push the red curves into it.  The two balls bounded by $S$ lift, in the branched cover, to handlebodies with (co-)attaching curves $\boldA$ and $\boldB$, respectively.  This Heegaard splitting is described by $(V, \boldA, \boldB(I^0))$. \end{proof}

Let $I$ be another resolution of $\diagram$.  We define a set of $2c$ curves $\boldB(I)$ by comparison with $\boldB(I^0)$. For each crossing where $I$ and $I^0$ agree, $\boldB(I)$ contains two curves parallel to those in $\boldB(I^0)$.  At crossings where they disagree, $\boldB(I)$ contains the green curves in Figure~\ref{fig:localbranchedpath}.  Place a basepoint $w$ on $S^2$ in the complement of $\diagram$ and all the little circles around crossings.  Write $\basew = \{w_1, w_2\}$ for the set of lifts of $w$ to $V$.

\begin{Def} 
  For any resolution $I$ of $\diagram$, the Heegaard diagram $\branched(I) = (V, \boldA, \boldB(I), \basew)$ is the \emph{branched diagram} of $\diagram(I)$.
\end{Def}

\begin{Lemma}
  The diagram $\branched(I)$ is a balanced, doubly-pointed Heegaard diagram which presents $S^3(\diagram(I^0))$.
\end{Lemma}
\begin{proof}
  The only thing to show is that $w_1$ and $w_2$ lie in disjoint regions in $V \setminus \left(\bigcup \boldA\right)$ and $V \setminus \left(\bigcup \boldB(I)\right)$ for every $I$, i.e.\ that the diagram is balanced.  Without loss of generality we may think of $w$ as sitting at $\infty \in S^2$ and assume that there is some closest crossing of $\diagram$ to $\infty$.  With this setup, it is easy to draw a curve on $V$ which connects $w_1$ and $w_2$ and intersects a single $\boldA$ curve only once.  A different curve connects the two basepoints and intersects a single $\boldB(I)$ curve only once.  This implies that the diagram is balanced.
\end{proof}

\begin{figure}[h]
\centering
 \def\svgwidth{.66\columnwidth}
 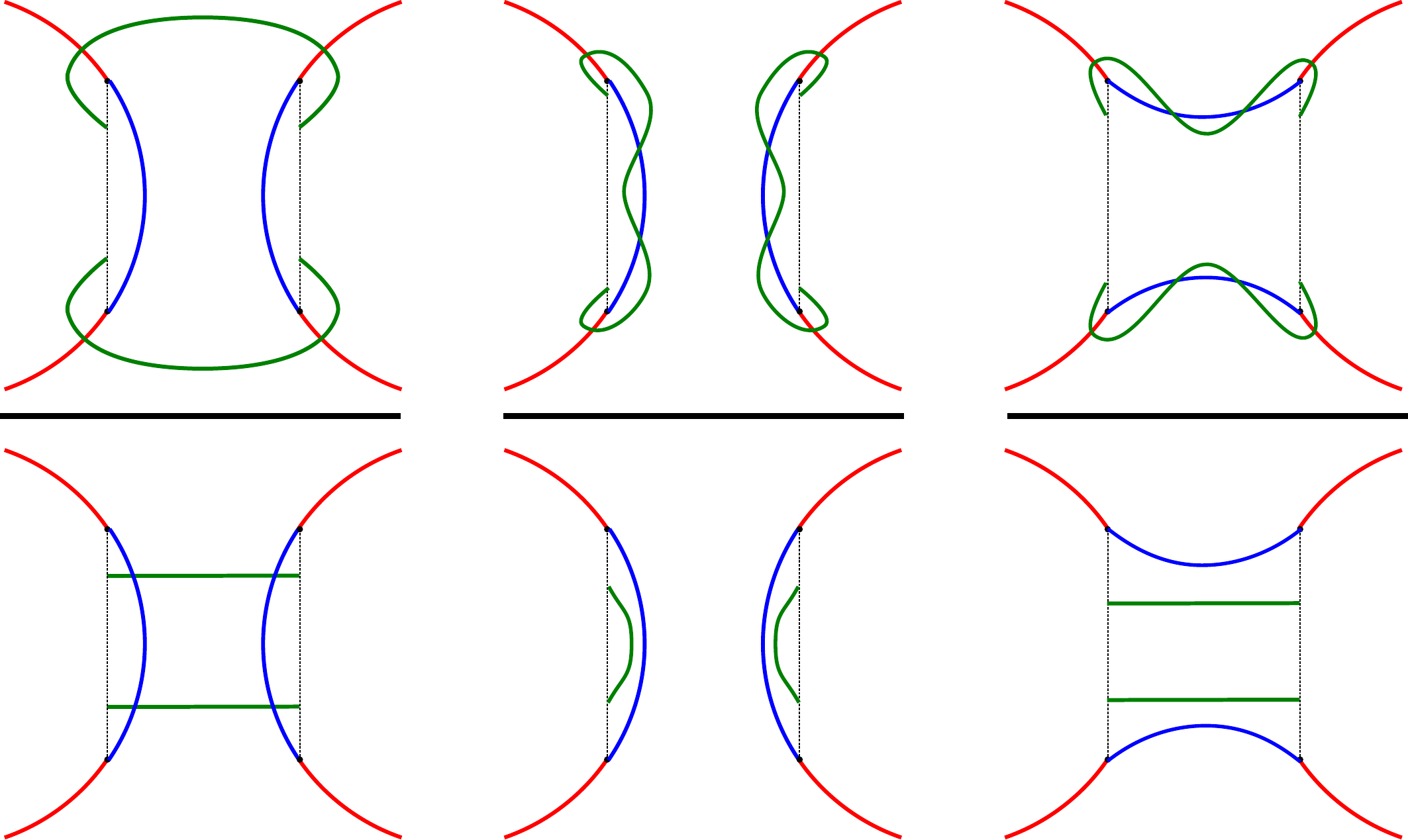
 \caption{Three different local pictures of a branched multi-diagram for a change of resolution.  On the left, the local picture at the changed resolution.  The other pictures are local pictures at unchanged crossings.  }\label{fig:localbranchedpath}
\end{figure}

\begin{Def} For a path of resolutions $\mathbf{I}$, the Heegaard multi-diagram \[
	\branched(\mathbf{I}) = (V, \boldA, \boldB(I_1),\ldots, \boldB(I_n), \basew)
\] 
is the \emph{branched multi-diagram} of $\mathbf{I}$.  \end{Def}

Note that $(V, \boldB(I_i), \boldB(I_{i+1}), \basew)$ is a standard Heegaard diagram for $\#^{2c} (S^2 \times S^1)$.  Thus there is a highest degree generator (and cycle) in $\CF(V, \boldB(I_i), \boldB(I_{i+1}), \basew)$.

\section{Bouquet diagrams from branched diagrams}\label{sec:bouquet}

Using branched diagrams, we can define a group $X' = \bigoplus_{I \in \{0,1\}^c} \CF(\branched(I))$ and a map $D': X' \to X'$ by analogy with $X$ and $D$.  It is not clear from the get-go that $D'$ is a differential, much less that the argument from Section \ref{sec:ozssz} produces a spectral sequence isomorphic to $\OS$.  Rather than attempting to adapt \Ozsvath and \Szabo's proof to this context, we will show that for every branched (multi-)diagram there is a sequence of handleslides which transform it into a bouquet (multi-)diagram.  Then we show that the quasi-isomorphisms of Heegaard Floer chain complexes induced by these handleslides induce an isomorphism of spectral sequences.\footnote{The Heegaard Floer homology of a three-manifold does not depend on the choice of Heegaard diagram: given two diagrams for $Y$  we can always connect them by a series of (pointed) isotopies, handleslides, and (de)stabilizations.  These moves induce natural isomorphisms on Heegaard Floer homology groups.  But it is not clear \emph{a priori} that they induce isomorphisms of spectral sequences, e.g.\ that they intertwine the entire differential on $X'$.}

Let $\diagram$ be a diagram of $L \subset S^3$ and form the diagram $(V, \boldA, \boldB(I^0))$.  The $\boldB(I^0)$ curves are paired together at crossings of $\diagram(I^0)$.  Let $\boldB'(I^0) \subset \boldB(I^0)$ be a choice of one such curve at each crossing.  Slide each element of $\boldB'(I^0)$ over the curve with which it is paired in $\boldB(I^0)$.  Diagrammatically, this replaces the chosen curve with a circle which contains the two branch cuts,  see Figure~\ref{fig:nearbybouquet}.  Call the collection of such circles $C$ and let $\bet(I^0) = C \cup \boldB'(I^0)$.  For any other resolution $I$, define $\bet(I) =  C \cup \boldB'(I)$ as follows: at a $0$-resolution, $\boldB'(I^0)$ and $\boldB'(I)$ contain parallel curves.  At a $1$-resolution, $\boldB'(I)$ instead contains a curve which intersects the corresponding curve in $\boldB'(I^0)$ once and does not intersect any other $\boldB'(I)$ curves.  

\begin{figure}[h]
\centering
 \includegraphics[width=.5\textwidth]{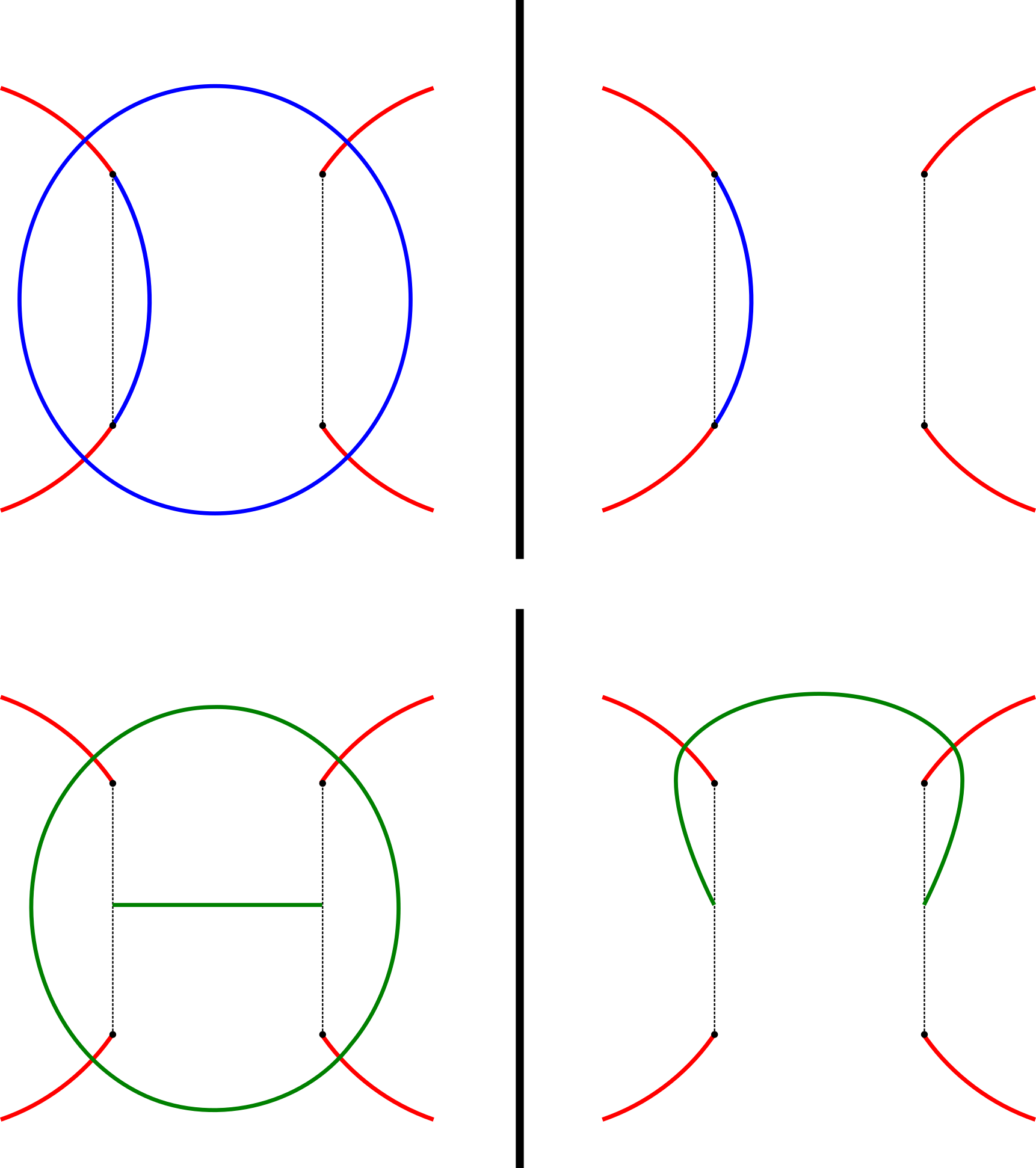}
 \caption{On top, both sides of $\bouquet(I)$ at a $0$-resolution.  On the bottom, both sides of $\bouquet(I)$ at a $1$-resolution.}\label{fig:nearbybouquet}
\end{figure}

\begin{Prop} The Heegaard diagram $\bouquet(I)$ presents $S^3(\diagram(I))$ and is a bouquet diagram for the surgery arcs from the \Ozsvath-\Szabo construction. \end{Prop}

\begin{proof} The diagram $\bouquet(I)$ differs from $\branched(I)$ by sliding the $C$ curve over the $\boldB'(I)$ curve at each crossing, so certainly $(V, \bal, \bet(I))$ presents $S^3(\mathpzc{D}_I)$.  At (the branched cover) of each crossing $\boldB'(I)$ consists of an element of $C$ and one of two other curves.  These latter two curves intersect each other once and do not intersect any other curves, so they lie on a punctured torus on $V$.  They are clearly a meridian and longitude for the lift of a surgery knot.
\end{proof}

For the remainder of this article, ``bouquet diagram'' will mean a diagram of the form $\bouquet(I)$.

\subsection{Admissibility}
\label{subsec:admissibility}
In this section we prove that all of these diagrams are weakly admissible in the sense of~\cite{Ozsvath2008}.  Let $I$ be a resolution of the link $L$.  On each side of $\branched(I)$ there are circles formed by alternating $\boldA$ and $\boldB$ curves.  Fixing our attention on one side of the diagram, these circles are in bijection with closed components of $\diagram(I)$.  For a component $Q$ of $\diagram(I)$, we say that a curve $\gamma$ in $\boldA$ or $\boldB$ \emph{belongs to $Q$} if $\gamma$ is a subset of $Q$ through this correspondence.

\begin{Prop}\label{periodicdomains}Branched diagrams are weakly admissible.\end{Prop}
\begin{proof} It suffices to show that every non-zero periodic domain on $\branched(I)$ has positive and negative coefficients.   Let $Q_1$ be an innermost component (relative to $w$) of $\diagram(I)$.  There are two regions bounded by the curves belonging to this component, and their difference is a periodic domain $P_1$.  Now let $Q_2$ be a component which contains $Q_1$ and does not contain any other component which contains $Q_1$.  There are two regions bounded by the curves belonging to $Q_1$ and the $\boldB$ curves belonging to $Q_0$, and their difference is a periodic domain $P_2$.  Note that the region obtained by instead using the $\boldA$ curves belonging to $C_0$ differs from this one by $\pm P_1$.  It is straightforward to extend this method to the remaining components of $\diagram$.  The result is a linearly independent set $\mathbf{P}$ of primitive periodic domains each corresponding to a closed component of $\diagram(I)$.  Each of these domains has both positive and negative coefficients.  The group of periodic domains on $\branched(I)$ is isomorphic to $H^1((S^1 \times S^2)^n; \bz)$, so $\mathbf{P}$ generates the group.  It is clear that any sum of these domains has both positive and negative coefficients.\footnote{The skeptic is encouraged to play the game \emph{Lights Out} from Tiger Electronics.} \end{proof}

Bouquet diagrams are admissible by a similar argument.  There is an obvious one-to-one correspondence between $\boldB$ curves and $\bet$ curves, and so we may talk about a curve in $\bouquet(I)$ belonging to a component of $\diagram(I)$.

\begin{figure}[h]
\centering
 \includegraphics[width=.66\textwidth]{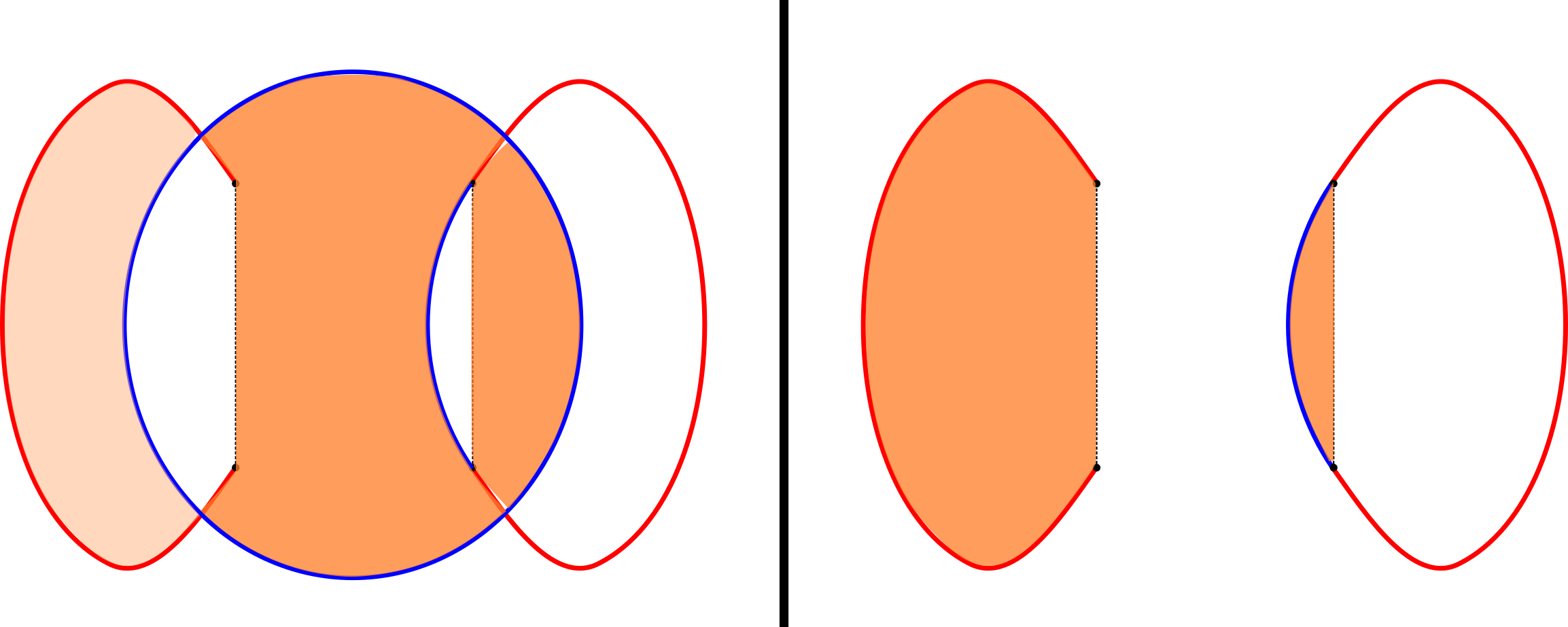}
 \caption{A bouquet diagram for the $0$-resolution of an unknot.  The lightly shaded region is the `obvious' bigon for the left component.  The darker region is topologically an annulus, but it can represent a holomorphic disk into the symmetric product of the Heegaard surface by cutting along a $\boldA$ curve.}
\label{fig:bouquetadmissible}
\end{figure}

\begin{Prop} Bouquet diagrams are weakly admissible. \end{Prop}
\begin{proof}
For each closed component of $\diagram(I)$, there are two regions bounded by curves belonging to that component.  If all the $\bet(I)$ curves belonging to a component are also in $\boldB'(I)$, then the regions are identical to those in the previous lemma.   When a curve in $C$ bounds a region, it is still easy to see it bounds one region.  The complementary region is more complicated; an example is shown in Figure~\ref{fig:bouquetadmissible}.  Again, the difference between these two regions is a periodic domain with positive and negative coefficients.
\end{proof}

Every other Heegaard diagram we consider consists of pairs of parallel curves.  In such a diagram, it is easy to find a basis of periodic domains with positive and negative coefficients and whose sums must have positive and negative coefficients provided that the parallel curves are properly perturbed.

\subsection{Na\"{i}ve correspondence with Khovanov chain group}\label{subsec:minimality}

\begin{Prop}\label{structurelemma}  Let $\branched(I)$ be a branched diagram for the resolution $I$ of a link diagram $\diagram$.   Then $\rank\CF(\branched(I)) = \rank\HF(\branched(I)) = \rank \CKh(m(L(I)))$.  In particular, the differential on $\CF(\branched(I))$ vanishes.  Also, the only $\Spinc$ structure on $\#\S^1 \times S^2$ which is represented by a generator in $\CF(\branched(I))$ is the unique torsion one. \end{Prop}

\begin{proof} 
The key observation is that there is a bijection between generators of $\CF(\branched(I))$ and orientations of the components of $\diagram(I)$.  Suppose that $\diagram(I)$ has $n$ closed components.  Then $\rank(\CF(\branched(I))) = 2^n$.  Meanwhile, the Heegaard diagram $\branched(I)$ presents $\#^{n-1} S^1 \times S^2$ where $n$ is the number of connected components of the resolved diagram.    So it's Heegaard Floer homology is that of $\#^n S^1 \times S^2$ and therefore it has rank $2^n$.  So the differential on $\CF(\branched(I))$ must vanish.

In~\cite{Ozsvath2004b} it is shown that the Heegaard Floer homology of $\#^n S^1 \times S^2$ is concentrated in the unique torsion $\Spinc$ structure $\fraks_0$.  Each generator must represent a non-trivial homology class, so they must all represent $\fraks_0$.
\end{proof}

For a three-manifold $Y$ and a torsion $\Spinc$ structure $\fraks \in \Spinc(Y)$, the group $\CF(Y,\fraks)$ has an absolute $\bq$-valued grading $\gr$ defined in~\cite{Ozsvath2006}.  We note here only a few properties of $\gr$.  First, for any Heegaard diagram of $S^1 \times S^2$ in which $\CF(S^1 \times S^2)$ has two generators, the generators have gradings $1/2$ and $-1/2$.  The assignment of gradings to intersection points depends crucially on the placement of the basepoint $z$.  Second, there is a K\"{u}nneth formula for Heegaard Floer homology: 	$\CF(Y_1 \# Y_2) \cong \CF(Y_1) \otimes_{\bz/2\bz} \CF(Y_2)$.  The isomorphism is of graded vector spaces.  In other words, the grading of a generator in $\CF(Y_1 \# Y_2)$ is the sum of gradings of the corresponding generators in $\CF(Y_1)$ and $\CF(Y_2)$.

Khovanov homology is a bigraded theory.  For a generator $x$ in a group corresponding to the resolution $I$, the \emph{homological grading} is given by $h(x) = |I| - n_-$.  Here $|I|$ is the sum of the entries of $I$, sometimes called the \emph{weight} of $I$, and $n_-$ is the number of negative crossings in $\diagram$.  The second grading is called the \emph{quantum grading} $q(x)$.  It is defined as follows: let $\tilde{q}(v_+) = -\tilde{q}(v_-) = 1$, and extend $\tilde{q}$ additively to simple tensors.  Then $q$ is given by 
\[
q(x) = \tilde{q}(x) + h(x) + n_+ - 2n_-
\]
where $n_+$ is the number of positive crossings in $\diagram$.

Let $h': X' \to \bz$ be defined identically to $h$.  Let $\tilde{q}' = 2 \gr$, and define $q':X' \to X'$ by
\[
q'(x) = \tilde{q}'(x) + h'(x) + n_+ - 2n_-.
\]

\begin{Prop}\label{prop:firstIsomorphism}  The correspondence in the proof of Proposition~\ref{structurelemma} is a graded isomorphism of the vector spaces $X'(\diagram)$ with bigrading $(h',q')$ and $\CKh(\diagram)$ with brigading $(h,q)$.\end{Prop}

\begin{proof}
We saw that generators of $\CF(\branched(I))$ are in one-to-one correspondence with orientations of the components of $\diagram(I)$.  The discussion above allows us to think of oriented component as having it's own Maslov grading.  The map $g_I \co \CF(\branched(I)) \to \CKh(\diagram(I))$ defined on canonical generators by sending higher Maslov-graded orientations to $v_+$ and lower-graded orientations to $v_-$ is clearly a $q$-graded isomorphism.  The sum $\bigoplus_I g_I$ obviously respects the homological gradings $h$ and $h'$.
\end{proof}

\section{Handleslides}\label{sec:handleslides}

For a link diagram $\diagram$ with $c$ crossings we have two groups:
\begin{align*}
X &= \bigoplus_{I \in \{0,1\}^c} \widehat{CF}(\mathpzc{Bo}(I))\\
X' &= \bigoplus_{I \in \{0,1\}^c} \widehat{CF}(\mathpzc{Br}(I)).
\end{align*}
The group $X$ is equipped with a filtered differential $D$.  The group $X'$ is equipped with a filtered map $D'$, defined identically \emph{mutatis mutandis}.  In this section we study the maps induced by the handleslides which transform $\mathpzc{Bo}(I)$ to $\mathpzc{Br}(I)$.

\begin{figure}[h]
\centering
 \includegraphics[width=.66\textwidth]{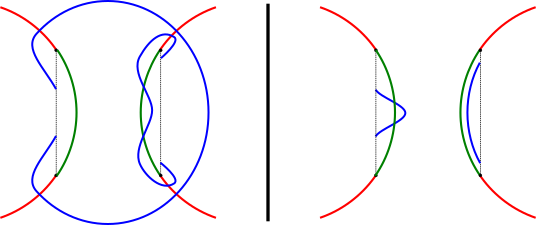}
 \caption{A branched multi-diagram representing a handleslide at a $0$-resolved crossing.}\label{fig:localbranchedhandleslide}
\end{figure}

To start, we order the crossings of $\diagram$ and examine the map induced by the handleslides near the first crossing.  As in~\cite{Roberts2008}, we are justified in drawing and sliding parallel curves together.  Figure~\ref{fig:localbranchedhandleslide} shows a Heegaard multi-diagram realizing a handleslide at a $0$-resolved crossing.  Call the handleslid curves $\boldB_1(I)$.  The local picture at a $1$-resolved crossing is similar.  At all other crossings of the multi-diagram, $\bet(I)$ and $\boldB_1(I)$ curves are parallel.  For any resolution $I$ the diagram $(V, \bet(I), \boldB_1(I))$ presents a connected sum of $S^1\times S^2$s and $\widehat{CF}(\bet(I), \boldB_1(I))$ has a highest degree generator $\Theta$.

Write $X_1$ for the sum of Heegaard Floer chain groups of the diagrams after a handleslide, i.e.\ using $\boldA$ and $\boldB_1(I)$.   We may view $X \oplus X_1$ as a $(c+1)$-dimensional cubical complex in which the last coordinate specifies whether to use curves from $\bouquet(I)$ (for $0$) or $\branched(I)$ (for $1$) at the first crossing.  For every path $\mathbf{I} = \{I_1, \ldots, I_n\}$ in the enlarged cube whose last coordinate changes define  $\psi_\mathbf{I} \co X \to X'$ by
\[
	\psi_\mathbf{I}(x) = f(x \otimes \Theta_1 \otimes \cdots \Theta_{n-1}).
\]
Here $\Theta_i$ is the highest graded generator in some diagram representing a connected sum of $S^1\times S^2$s.  Let $\psi_1 = \sum \psi_\mathbf{I}$ where the sum is taken over paths whose last coordinate changes.  Define maps $D_0$ and $D_1$ by
\begin{align*}
	D_0 &= \sum_{J, J' \in \{0,1\}^k \times \{0\}} d_{J,J'} \\
	D_1 &= \sum_{J, J' \in \{0,1\}^k \times \{1\}} d_{J,J'}
\end{align*}
\begin{Prop}\label{prop:equivariant}
The map $\psi_1$ satisfies $\psi_1 \circ D_0 = D_1 \circ \psi_1$.
\end{Prop}

\begin{proof}
Let $\phi \in \pi_2(\mathbf{x}, \Theta_1, \ldots, \Theta_{n-1})$ with $\mu(\phi) = 1$.  Write $\mathcal{M}(\phi)$ for the space of holomorphic representatives of $\phi$.  This space has a compactification whose ends are pairs $(\phi_1, \phi_2)$ so that $\phi_1 \star \phi_2 = \phi$ and $\mu(\phi_1) = \mu(\phi_2) = 0$ where $\star$ is the concatenation operation.  More concretely, the ends are ``broken polygons'' joined at a vertex.  Each of these corresponds to a holomorphic polygon with $\mu = 1$ and a degenerating chord, as shown in Figure~\ref{fig:degeneration}.  Compactness of $\mathcal{M}(\phi)$ implies that the sum of the ends must be even.  In this schematic, each edge of the polygon is mapped to a set of attaching curves (really, to a torus in a certain symmetric product) and we will abuse notation by identifying the side of the polygon with the curves.

\begin{figure}[h]
\centering
 \includegraphics[width=.66\textwidth]{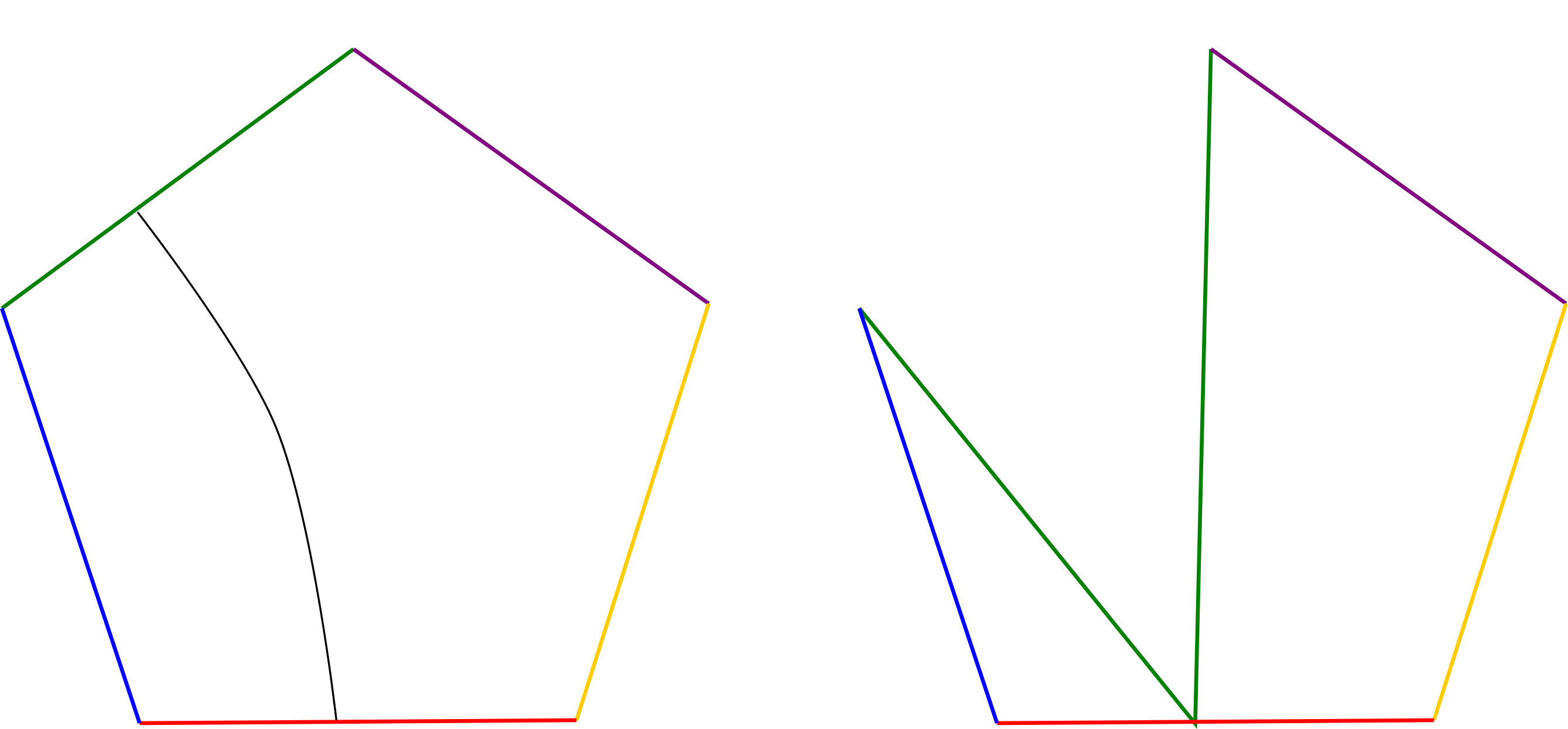}
 \caption{A holomorphic pentagon degenerates along a chord into a holomorphic triangle and rectangle.  Supposing that the vertex $v$ is where the green and purple edges meet, this degeneration is counted by $\psi_1 \circ D_0$}
\label{fig:degeneration}
 \end{figure}

Each polygon counted by $\psi_1$ has a unique vertex $v$ joining a $\bet$ edge and a $\boldB_1$ edge which represents the handleslide. The character of a degeneration is determined by the position of the degenerating chord relative to $v$ and to the $\boldA$ edge.  Suppose the chord has an endpoint on the $\boldA$ edge.  The other endpoint must be either to the left or to the right of $v$, so in the degeneration $v$ ends up in one polygon or the other.  The polygon with $v$ is counted by $\psi_1$.  The other polygon is counted by either $D_0$ or $D_1$.   Summing over all degenerations, every concatenation counted by $\phi_1 \circ D_0$ and $D_1 \circ \psi_1$ appears exactly once.  So degenerations along a chord touching the $\boldA$ edge contribute $\psi_1 \circ D_0 + D_1 \circ \psi_1$ to sum of boundary components.

Suppose that the degenerating chord does not have an endpoint on $\boldA$ and also that it does not separate $\boldA$ from $v$.  After degenerating, one polygon is made entirely of either $\bet$ or $\boldB_1$ curves, so the corresponding polygon must connect standard generators.  The cancellation lemma, Lemma 4.5 of~\cite{Ozsvath2005b} and Lemma 7 of~\cite{Roberts2008}, shows that the maps counting such polygons sum to zero.  The following lemma show the same in the case in which the degenerating chord separates $v$ from $\boldA$.  This completes the proof, as it follows $\psi_1 \circ D + D_1 \circ \psi_1$ is equal to the sum of all the boundary components, and therefore to zero in $\bf$.   \end{proof}
\begin{Lemma}
\label{cancellationlemma}
Let $I, I'$ be elements of $\{0,1\}^k \times \{0,1\}$ with differing last coordinates so that $I < I'$.  Then for any $n \leq k$, 
\[
	\sum_{\text{$\mathbf{J} \in \mathcal{P}(I,I')$}} f_\mathbf{J}(\Theta_1 \otimes \cdots \otimes \Theta_{n}) = 0.
\]
\end{Lemma}
\begin{proof}
Our argument closely follows those in~\cite{Ozsvath2005b} and~\cite{Roberts2008}.  We start with the case that $n > 2$.  We will show that any element of $\pi_2(\Theta_1, \ldots, \Theta_n)$ has positive Maslov index and therefore cannot contribute to the sum.

One may construct a polygon counted in the sum by splicing together several triangles as in Figure~\ref{fig:concatenating}.  There are three sorts of constituent triangles:
\begin{enumerate} \item those of the form $(\bet(J_0), \bet(J_i), \bet(J_{i+1}))$
 \item those of the form $(\bet(J_0), \boldB_1(J_j), \boldB_1(J_{j+1}))$
 \item a single triangle of the form $(\bet(J_0), \bet(J_k), \boldB_1(J_{k+1}))$.
\end{enumerate}

\begin{figure}[h]
\includegraphics[width=.66\textwidth]{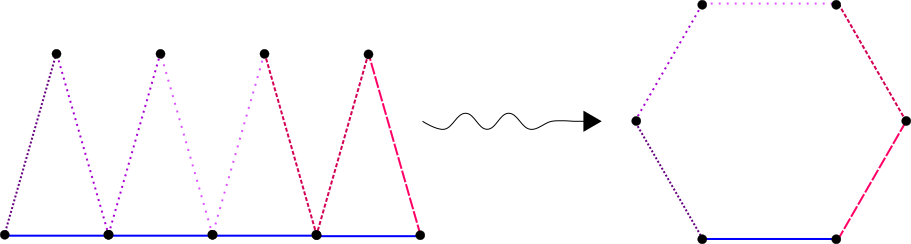}
\caption{Concatenating four triangles and un-degenerating them to obtain a hexagon.  The blue and purple curves are type $\bet$ and the red curves are type $\boldB_1$.}
\label{fig:concatenating}
\end{figure}

The third type contains the vertex $v$.  Every triple of the first type presents a surgery cobordism between connected sums of $S^1 \times S^2$s.  Any contributing triangle must have Maslov index $0$, so the corresponding map of Floer homologies sends the highest degree generator in $\widehat{CF}(\bet(J_0), \bet(J_i)) \otimes \widehat{CF}(\bet(J_i), \bet(J_{i+1}))$ to the highest degree generator in $\widehat{CF}(\bet(J^0), \bet(J^{i+1}))$.  (See Lemma 3 in~\cite{Roberts2008}.)  To obtain an honest polygon we must `undo' the degenerations, increasing the Maslov index by $1$ for each splicing.  As all the $\Theta$s live in torsion $\Spinc$ structures, the addition of a periodic domain to this polygon does not affect its Maslov index.  Recall that any two domains representing the homotopy class of a Whitney polygon which might be counted in the differential differ by a periodic domain.  It follows that \emph{every} polygon between these generators has Maslov index greater than $0$.

The case $n = 1$ is equivalent to the fact that $\Theta$ is a cycle.  In the case $n = 2$, the map $f_\mathbf{J}$ counts certain triangles in which one vertex is $v$.  These correspond to changes of resolution of the form $(0, \ldots; 0) \to (1, \ldots; 0) \to (1, \ldots; 1)$ or $(0, \ldots; 0) \to (0, \ldots; 1) \to (1, \ldots; 1)$.  In each case there is a unique triangle connecting highest degree generators, see Figure~\ref{fig:branchedcancellation}.  The triangles shown have Maslov index $0$ (see, for example, the main result of~\cite{Sarkar2011}).  An analysis of periodic domains shows that there are no other triangles between the highest degree generators.  The remaining curves in the diagram come in parallel triples.
\end{proof}

\begin{figure}[h]
\centering
\begin{subfigure}[h]{.4\textwidth}
 \includegraphics[scale=.5]{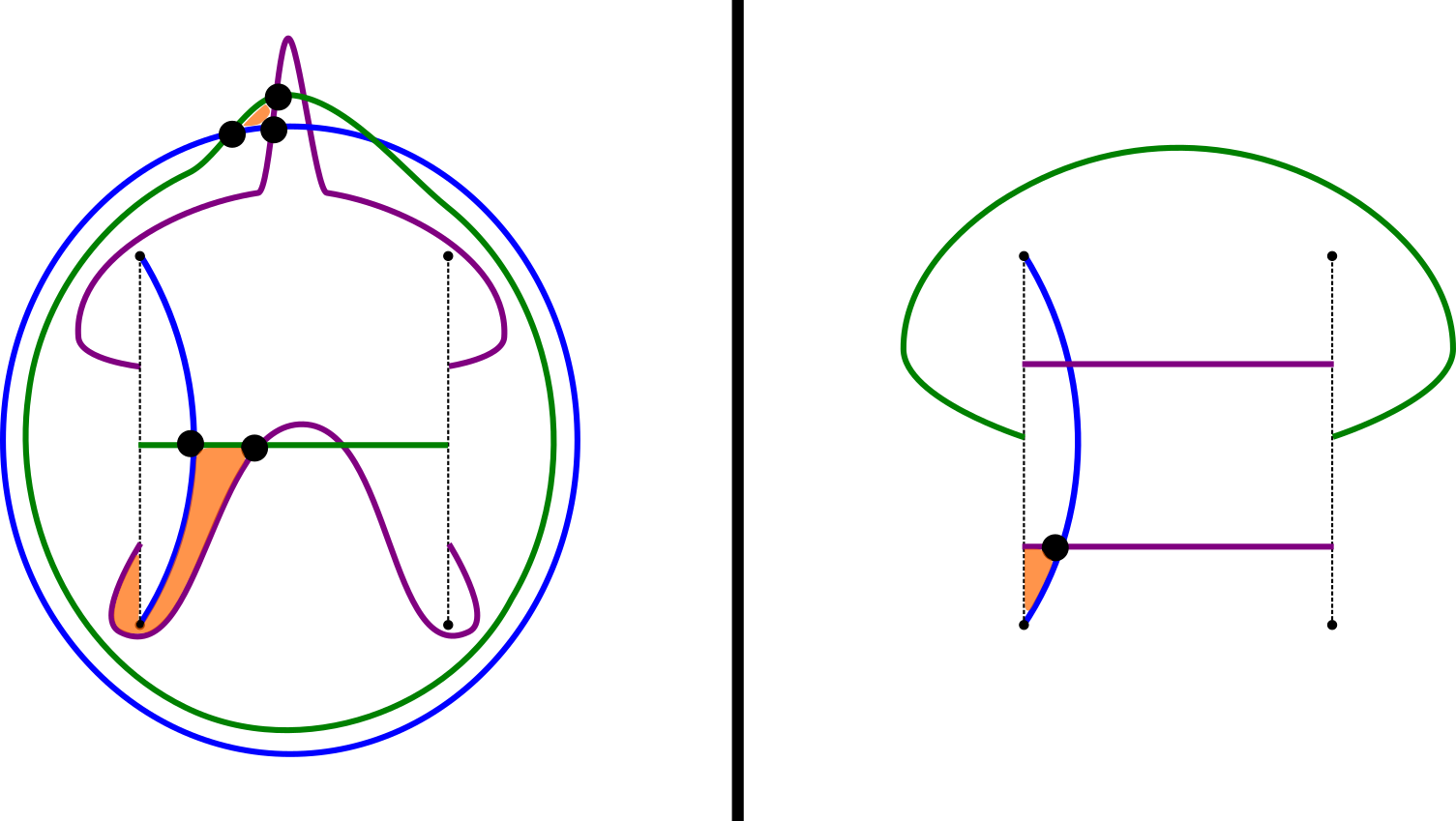}
\end{subfigure}
\hfill
\begin{subfigure}[h]{.4\textwidth}
\includegraphics[scale=.5]{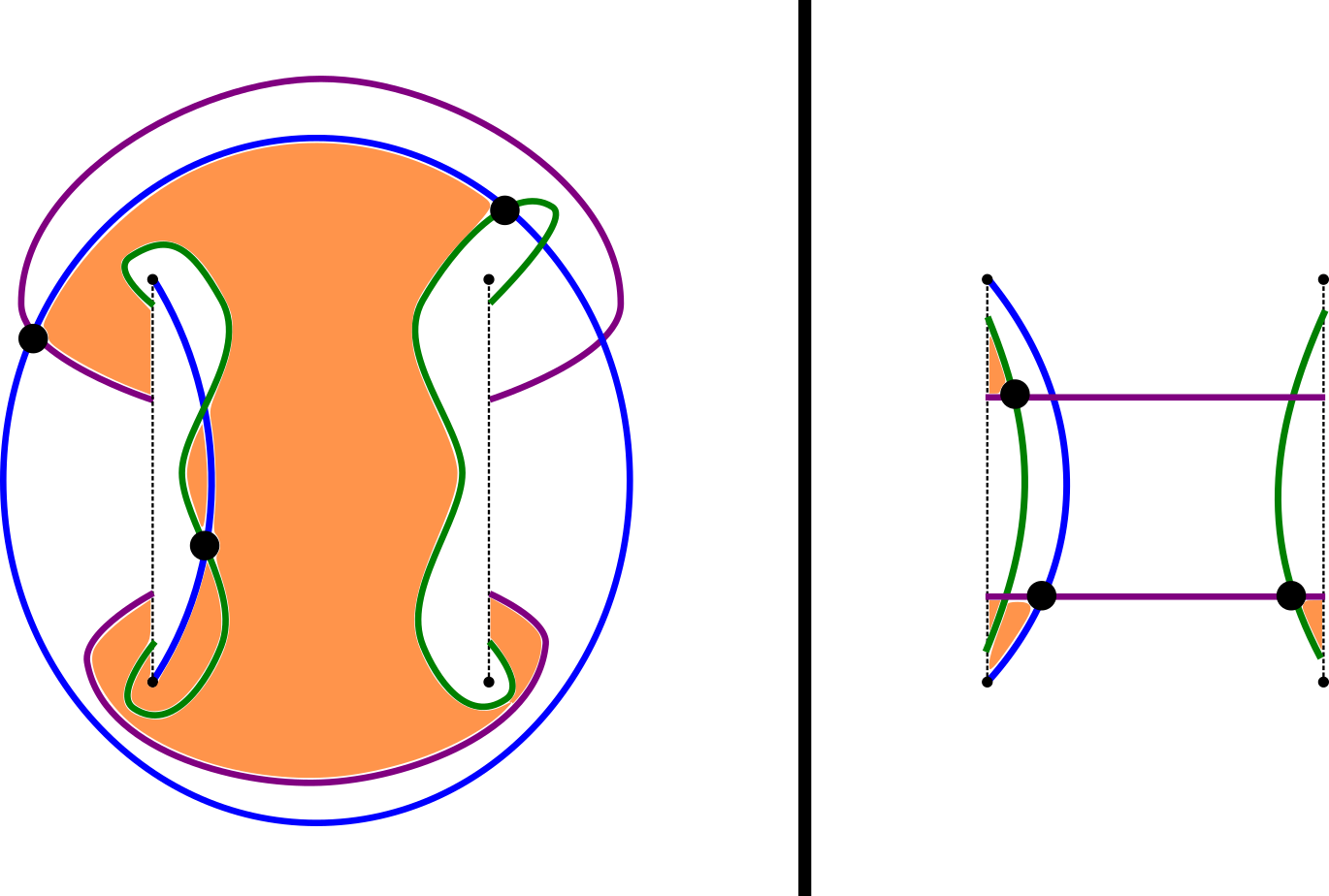}
\end{subfigure}
\caption{The $\boldB$ curves for the paths $(0, \ldots; 0) \to (1, \ldots; 0) \to (1, \ldots; 1)$ (left) and $(0, \ldots; 0) \to (0, \ldots; 1) \to (1, \ldots; 1)$ (right).  The intersection points constituting the highest degree generators are marked with circles. The shaded region is the region of the unique holomorphic triangle connecting the generators.}
\label{fig:branchedcancellation}
\end{figure}

This argument is essentially local, so it works just as well if we change a second crossing after changing the first and so on.  Let $X_i$ be the group obtained by doing handleslides at the first $i$ crossings.  Let $D_i$ be the putative differential on $X_i$, and let $\psi_i: X_i \to X_{i+1}$ be the map induced by the next handleslide.  Let $\Psi \co X \to X'$ be the map 
\[
	\Psi = \psi_n \circ \cdots \circ \psi_1.
\]
Write $\phi_i: X_i \to X_{i-1}$ for maps defined identically to $\psi_i$ but induced by the reverse handleslide.  Define $\Phi \co X' \to X$ by 
\[
	\Phi = \phi_1 \circ \cdots \circ \phi_n.
\]

\begin{Lemma} The pair $(X', D')$ is a chain complex.\end{Lemma}
\begin{proof}  It follows by iterating Proposition~\ref{prop:equivariant} that the maps $\Psi$ and $\Phi$ are $D$-$D'$-equivariant.  Therefore $\Phi \circ D'^2 = D^2 \circ \Phi = 0$.

Let $\phi_i^0$ be the component of $\phi_i$ which only counts paths of length two (i.e.\ holomorphic triangles). Let $\Phi^0 = \phi^0_2 \circ \cdots \circ \phi^0_n$ so that $\Phi^0$ is a direct sum of maps $\Phi_I^0 \co \CF(\branched(I)) \to \CF(\bouquet(I))$.  The proof of topological invariance of Heegaard Floer homology shows that each of these maps is a quasi-isomorphism.  Now Lemma~\ref{structurelemma} implies that $\Phi^0$ is injective: if $\rank(\phi^0_I) < \dim(\CF(\branched(I)))$, then $\rank((\phi^0_I)_*) < \dim(\CF(\branched(I))) = \dim(\HF(\branched(I)))$ and $\Phi^0_I$ could not be a quasi-isomorphism.

For the sake of contradiction,\footnote{This argument is essentially that of Lemma~\ref{lem:filteredisomorphism}, which is standard.  But as we have not yet proven that $X'$ is a chain complex, we record the argument here.} suppose that $(D')^2(\mathbf{x}) \neq 0$ for some $\mathbf{x} \in X$'.  Then $(D')^2(\mathbf{x})$ may have components in several direct summands.  One (or more) of these summands will be minimal in the order on the cube.  Let $x_0$ be a non-zero component of $(D')^2(\mathbf{x})$ in such a summand.  Then $\Phi \circ (D')^2(\mathbf{x})$ has a component $\Phi^0 \circ (D')^2 (\mathbf{x}) \neq 0$ in $\CF(\bouquet(I_0))$.  Because $x_0$ is in a minimal summand, no other component of $\Phi$ can cancel with $\Phi^0$, so $\Phi \circ (D')^2(\mathbf{x}) \neq 0$.  This contradiction implies that $(D')^2 = 0$.
\end{proof}

The complex $X'$ is filtered by the partial ordering on the cube.  We call the induced spectral sequence $\BR(\diagram)$.  Now we are able to complete the proof of the theorem stated in the introduction. 

\mainthm*

\begin{proof} The map $\Psi$ is a filtered chain map, so it induces a map of spectral sequences.  The induced map on $E^1$ is the direct sum of handleslide maps $\widehat{HF}(\bouquet(I)) \to \widehat{HF}(\branched(I))$ which are shown to be isomorphisms in the proof of the topological invariance of Heegaard Floer homology.  This implies that the maps induced by $\Psi$ are isomorphisms for $k \geq 1$, see Theorem 3.4 in~\cite{McCleary2000}.
\end{proof}

\begin{Cor}  The vector spaces $\BR^k(\diagram)$ are smooth link invariants for $k \geq 2$. \end{Cor}

\begin{proof} In~\cite{Baldwin2011}, Baldwin shows that $\OS^k(\diagram)$ is a (smooth) link invariant for $k \geq 2$.  The proofs of these facts translate to the branched setting.  The essential point is that if $\diagram$ and $\diagram'$ are diagrams of $L$, any sequence of Reidemeister moves induces a map $X'(\diagram) \to X'(\diagram')$ which in turn induces an isomorphism $\BR^2(\diagram) \cong \BR^2(\diagram')$.  Theorem 3.4 in~\cite{McCleary2000} implies that the two spectral sequences are isomorphic.
\end{proof}

\section{\Szabo homology, knot Floer homology, and a new conjecture}\label{sec:szabo}

In~\cite{Szabo2013}, \Szabo defines a combinatorial link homology theory which is conjecturally equivalent to the Heegaard Floer homology of the branched double cover.  In this section we compare this theory to the spectral sequence $\BR$ in an attempt to prove Conjecture \ref{conj:original}.  The two theories are not the same, and the discrepancy between them leads us to a variation of \Szabo's theory for link diagrams on a sphere with two basepoints, i.e.\ annular link diagrams.  We conjecture that the analogous theory on the Heegaard Floer side is a version of knot Floer homology.  The evidence for this conjecture (and the motivation for this section) is Proposition~\ref{prop:agreement}, which shows that the conjecture holds for diagrams with two crossings.

\subsection{Brief review of Khovanov homology}\label{subsec:kh}

First we recall the definition of Khovanov homology from~\cite{khovanov2000} to set some notation.  Let $\diagram$ be a spherical link diagram.  Let $V = \Span_{\bF}\{v_+, v_-\}$ for formal symbols $v_+$ and $v_-$.  Let $J$ be a resolution of $\diagram$ so that $\diagram(J)$ has $\|J\|$ closed components, and define
\[
	\CKh(\diagram(J)) = V^{\otimes \|J\|}.
\]
Suppose that $\diagram$ has $c$ crossings.  Define
\[
	\CKh(\diagram) = \bigoplus_{J \in \{0,1\}^c} \diagram(J).
\]
Suppose that $I$ and $J$ are connected by an edge in the cube of resolutions, i.e. they differ only in one coordinate.  Then $\diagram(I)$ and $\diagram(J)$ in one of two ways.  Either $\diagram(J)$ is obtained by merging two components of $\diagram(I)$ into one, or $\diagram(J)$ is obtained by splitting a component of $\diagram(I)$.  Call these components \emph{active} and the remaining components \emph{passive}.  Define
\[
	\partial_1(I,J) \co \CKh(\diagram(I)) \to \CKh(\diagram(J))
\]
as follows.  On each tensor factor of $\CKh(\diagram(I))$ corresponding to a passive circle, $\partial_1(I,J)$ acts as the identity.  (More precisely, it identifies the factor with the corresponding factor in $\CKh(\diagram(J))$ by the identity.)  On the active circles, $\partial_1(I,J)$ acts by a map $m \co V \otimes V \to V$ for a merge or $\Delta \co V \to V \otimes V$ for a split.

Now we can define the differential on Khovanov homology:
\begin{align*}
&\partial_1 \co \CKh(\diagram) \to \CKh(\diagram)\\
&\partial_1 = \sum_{I \prec J} \partial_1(I,J)
\end{align*}
The chain complex has two gradings, $h$ and $q$.
\begin{align*}
h &= |I| - n_- \\
\tilde{q}(v_\pm) &= \pm1 \\
q &= \tilde{q} + h + n_+ - 2n_-
\end{align*}

\begin{Thm*}[\cite{khovanov2000}]
  Let $\diagram$ be a link diagram for $L \subset S^3$.  The pair $(\CKh(\diagram),\partial_1)$ is a chain complex with bigrading $(h,q)$.  Its bigraded homology $\Kh(L)$ (and in fact its chain homotopy type) is a smooth link invariant.
\end{Thm*}

\subsection{\Szabo homology and the discrepancy with $\BR$}\label{subsec:szabo}

\begin{Def} Let $C_0$ be an embedded collection of circles in $S^2$.  Let $\decos$ be a collection of $i \geq 1$ oriented arcs embedded in $S^2$ so that $\partial\decos \subset C_0$ and so that the interior of each arc is disjoint from $C_0$.  Each arc is a \emph{decoration}.  The pair $(C_0, \decos)$ is a \emph{$i$-dimensional configuration}.  By doing surgery on each decoration in $\config$ we obtain a new collection of circles $C_1$.  We sometimes use the functional notation $\config \co C_0 \to C_1$.  The components of $C_0$ are called \emph{starting circles} and the components of $C_1$ are called \emph{ending circles}.  The circles in $C_0$ which meet the decorations are called \emph{active}.  The remaining circles are called \emph{passive}.
\end{Def}

\Szabo defines a linear map $\cmap_{\config} \co \CKh(C_0) \to \CKh(C_1)$ which depends on the configuration. The maps are rigged to follow the rules described below.  Let $\config \co C_0 \to C_1$ be a configuration.

\begin{enumerate}
\item If $\config'$ is the image of $\config$ under an orientation-preserving diffeomorphism of $S^2$, then $\cmap_\config = \cmap_{\config'}$ after identifying the components of $\config$ and $\config'$ by the diffeomorphism.
\item If the union of the active circles and their decorations is not connected, then $\cmap_\config = 0$.
\item On tensor factors corresponding to passive circles, $\cmap_\config$ acts as the identity.
\item $\tilde{q}(\cmap_\config(x)) - \tilde{q}(x) = i - 2$. (This is the \emph{grading rule}.)
\item $\cmap_\config = \cmap_{r(\config)}$ where $r(\config)$ is $\config$ with oppositely oriented decorations.
\item There is a dual configuration $\config^* \co C_1 \to C_0$ whose decorations are the images of the decorations of $C$ rotated counterclockwise by $90^\circ$.  Let $m(C^*)$ be the image of this configuration under the antipodal map on $S^2$.  There is also duality endomorphism on $\CKh(\diagram)$ defined on canonical generators by swapping $+$ and $-$ signs.  Write $x^*$ for the dual of $x$.

Let $x$ be a canonical generator of $\CKh(C_0)$ and $y$ a canonical generator of $\CKh(C_1)$.  Then the coefficient of $F_{\config}(x)$ at $y$ is equal to the coefficient of $F_{m(\config^*)}(y^*)$ at $x^*$.
\item Let $p \in C_0$.  Given $\config \co C_0 \to C_1$ we can identify $p$ with a point $p' \in C_1$.  There is a map $x_p: \CKh(C_0) \to \CKh(C_0)$ acting by $x_p (v_+) = v_-$ and $x_p (v_-) = 0$ on the factor assigned to the component containing $p$ and the identity elsewhere.  The map $x_{p'}$ is defined similarly.  The rule states that $\mathcal{F}_{C}(\im x_p) \subset  \im(x_{p'}) $.
\end{enumerate}

The exact form of these maps can be found in~\cite{Szabo2013}.

\begin{Rem} Many of these rules have counterparts in Heegaard Floer theory.  For example, rule (3) is analogous to the behavior of a polygon-counting map in a multi-diagram whose curves are all parallel.  The grading rule corresponds to the \emph{grading conjecture} that there is a grading on the \Ozsvath-\Szabo spectral sequence so that the differential on $E^k$ acts with grading $2 - k$, see the introduction of~\cite{Baldwin2011} or the last section of~\cite{Greene2013}.  The duality rule has analogues in Heegaard Floer theory (Theorem 3.5 of~\cite{Ozsvath2004b}) and Khovanov homology (see Bar-Natan's reformulation~\cite{BarNatan2005}, in which the duality rule can be understood as turning a cobordism upside-down).  The last rule rule allows for the definition of a reduced theory and seems to be related to the dotted cobordisms in \cite{BarNatan2005}.
\end{Rem}

Let $\diagram$ be a spherical link diagram for a link $L$.  The bigraded vector space underlying $\CSz(\diagram)$ is $\CKh(\diagram)$.   Choose orientations for decorations $\mathbf{t}$ on $\diagram(I^0)$.  This decorates every other diagram in the cube of resolutions for $\diagram$ under the convention that change of resolution rotates decorations counterclockwise by $90^\circ$.  Let $F$ be a $k$-dimensional face in the cube of resolutions with initial face $\CSz(\diagram(I))$ and final face $\CSz(\diagram(J))$.  The difference between $\diagram(I)$ and $\diagram(J)$ is encoded by $k$ decorations.  This defines a configuration $\config(I,J)$ and thus a map
\[
  \cmap_{\config(I,J); \mathbf{t}} \co \CSz(\diagram(I)) \to \CSz(\diagram(J)).
\]
For $i>1$, let
\[
  \partial_{i; \mathbf{t}} \co \CSz(\diagram) \to \CSz(\diagram)
\]
be the sum of all maps along the $i$-dimensional faces.  Let $\partial_{1;\mathbf{t}}$ be the Khovanov differential.  Define
\[
  \partial_{\mathbf{t}} = \sum_{i\geq 1} \partial_{i; \mathbf{t}} =
\]

\begin{Thm*}[\Szabo]
  The pair $(\CSz(\diagram), \partial_\mathbf{t})$ is a bigraded chain complex.  It's bigraded chain homotopy type and homology are link invariants.

  The chain homotopy type of $\Sz(\diagram, \mathbf{t})$ does not depend on the orientations $\mathbf{t}$.  That   is, if $\mathbf{t}$ and $\mathbf{t}'$ are distinct orientations for the decorations on $\diagram(I^0)$, then $\CSz  (\diagram, \mathbf{t}) \simeq \CSz(\diagram, \mathbf{t}')$.
\end{Thm*}

The order filtration on the cube of resolutions induces a spectral sequence $\ESz(\diagram)$ from Khovanov homology to \Szabo homology.  Based on structural similarities and Seed's computations~\cite{Seed2011}, Seed and \Szabo conjectured that this is the same as \Ozsvath and \Szabo's.

\seedszabo*

We started this project in the hope of showing that $\BR(\diagram) \cong \ESz(\diagram)$, but this is not true in general.  Consider the configuration shown in Figure~\ref{fig:configuration1}.  Figure~\ref{fig:notriangle} shows the corresponding branched multi-diagram.  If the spectral sequences agreed, there would be a single (mod 2) holomorphic quadrilateral representing the map $\mathcal{F}_{C_2}$.  Figure~\ref{fig:notriangle} shows that there are no such quadrilaterals: the region corresponding to such a quadrilateral would differ from the region shown by a periodic domain, and the resulting domain would have negative coefficients.  In fact, the only non-zero component of the Heegaard Floer map from the multi-diagram in Figure~\ref{fig:notriangle} sends $v_+ \otimes v_+$ to $v_+ \otimes v_-$, where in the latter tensor product the labeling of the outer circle comes first.  This map cannot be part of \Szabo's theory as it violates the grading rule.  (The map assigned by \Szabo is $v_+ \otimes v_+ \mapsto v_+ \otimes v_+$.)

\begin{figure}[h]
\includegraphics[width=.25\textwidth]{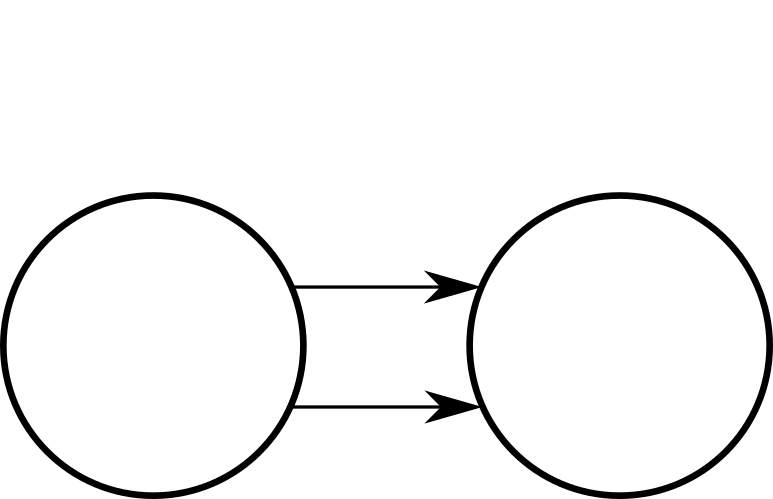}
\caption{Configuration 1 in \Szabo's labeling.}
\label{fig:configuration1}
\end{figure}

Further examination of the two-dimensional configurations suggests that the discrepancy between $\BR$ and $\ESz$ stems from the spherical symmetry in $\ESz$ which is not present in $\BR$.  In $\ESz$, the two ending circles of Configuration 1 are indistinguishable: they may be isotoped past $\infty$ to appear concentric.  This is not possible in the Heegaard Floer world due to the basepoints.  In the next section we add some basepoints to \Szabo's theory to compensate.

\begin{figure}[h]
\includegraphics[width=\textwidth]{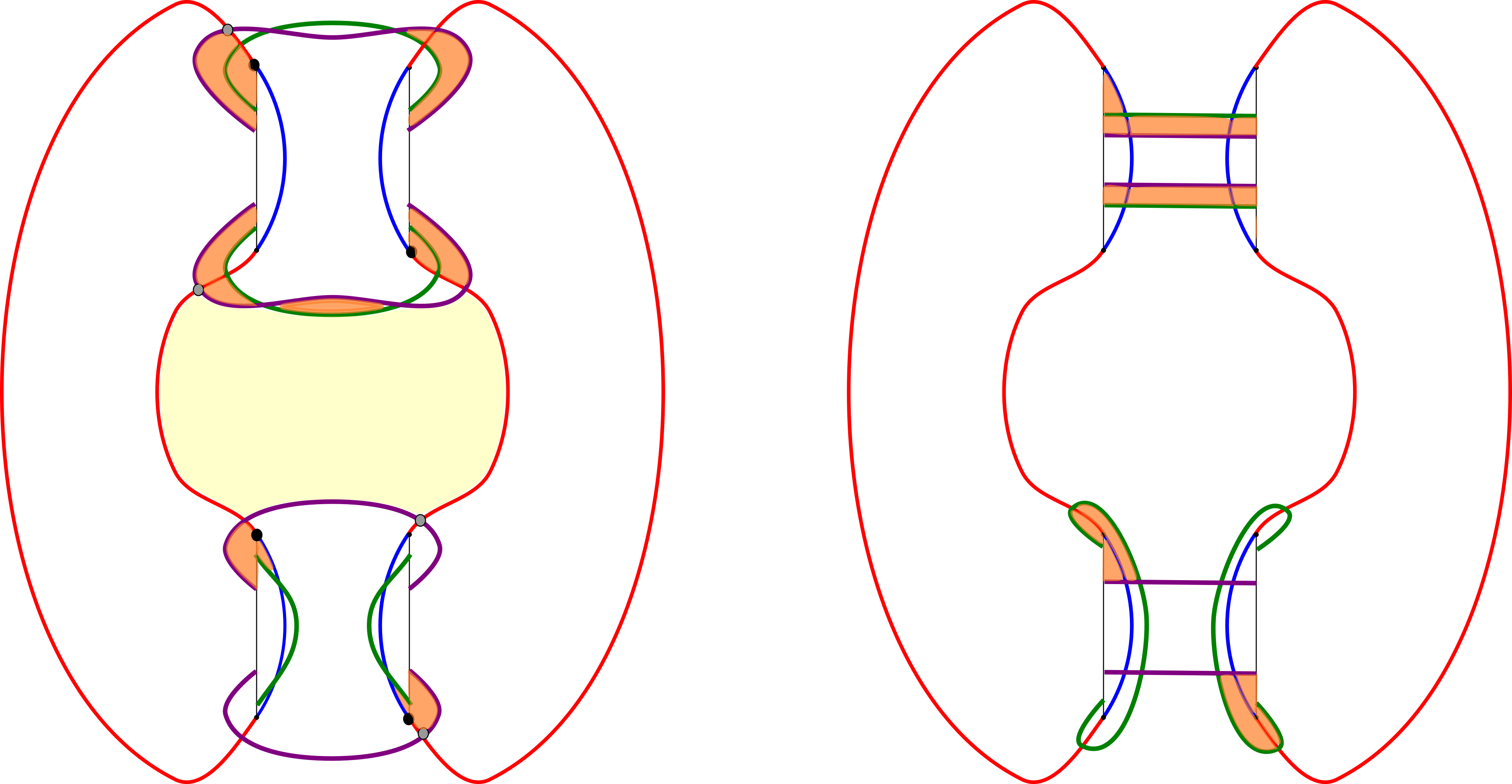}
\caption{The branched Heegaard multi-diagram from configuration 1.  The black (resp. gray) dots mark intersection points which correspond to the Heegaard Floer generators corresponding to $v_+ \otimes v_+$ on the starting (resp. ending) circles as in Section~\ref{subsec:minimality}.  The colored region corresponds to a Whitney quadrilateral.  The orange regions have coefficient $+1$ and the yellow region has coefficient $-1$.}
\label{fig:notriangle}
\end{figure}

\begin{Rem}
Jacobsson~\cite{Jacobsson2004}, Khovanov~\cite{Khovanov2006}, and Bar-Natan~\cite{BarNatan2005} independently showed that Khovanov homology with $\bz/2\bz$-coefficients is functorial under cobordism, in the sense that if $L$ and $L'$ are links in $S^3$ and $W$ is a cobordism from $L$ to $L'$ in $S^3 \times I$, there is a non-trivial map $F_W \co \Kh(L) \to \Kh(L')$ which depends only on the isotopy class of $W$ rel boundary and which can be computed from a diagrammatic description.  Heegaard Floer homology assigns maps to cobordisms of three-manifold: a cobordism $\tilde{W} \co S^3(L) \to S^3(L')$ induces a map $F_{\tilde{W}} \co \HF(S^3(L)) \to \HF(S^3(L'))$.  In forthcoming work the author and Baldwin show that \Szabo's theory is functorial, and so it is reasonable to conjecture that these maps on \Szabo homology agree with those on Heegaard Floer homology, see ~\cite{Baldwin2011}.
\end{Rem}

\subsection{A doubly-pointed link homology theory}

Let $\diagram$ be a spherical link diagram with $c$ crossings.  Let $z, w \in S^2 \setminus \diagram$ be two basepoints, not necessarily in distinct components.  For a resolution $J$, each component of $\diagram(J)$ represents either a trivial or non-trivial class in $H_1(S^2 \setminus \{z,w\})$.  Suppose that $\diagram(J)$ has $\ell$ trivial and $m$ non-trivial components.  Let $W$ be the vector space generated by the symbols $w_+$ and $w_-$ as in~\cite{Grigsby2011}.  Define
\begin{equation}\label{eqn:pointedCSz}
  \CSz_{z,w}(\diagram(J)) = W^{\otimes \ell} \otimes V^{\otimes m}.
\end{equation}
To be more precise, each non-trivial component is assigned a factor of $W$ and each trivial factor is assigned a factor of $V$.  Let 
\[
\CSz_{z,w}(\diagram) = \bigoplus_{J \in \{0,1\}^c} \CSz_{z,w}(\diagram(J)).
\]
The $h$-grading on $\CSz(\diagram)$ extends immediately to $\CSz_{z,w}(\diagram)$.  After defining $\tilde{q}(w_\pm) = \tilde{q}(v_\pm)$, the $q$-grading extends too.  There is a third grading $k$ defined on generators by $k(v_\pm) = \pm 1$ and $k(w_\pm) = 0$.  On Khovanov homology this grading is called the \emph{annular grading} and has been studied extensively, see for example~\cite{Roberts2013},~\cite{Grigsby2013}.

Define $\tilde{\sigma} \co W \to V$ by
\begin{align*}
  \tilde{\sigma}(w_\pm) &= v_\mp \\
\end{align*}
and linearity.  For any resolution $J$, this induces a map $\sigma(J)_{z,w} \co \CSz_{z,w}(\diagram(J)) \to \CSz(\diagram(J))$; in the decomposition of (\ref{eqn:pointedCSz}),
\[
  \sigma_{z,w}(J) = \tilde{\sigma}^{\otimes \ell} \otimes \Id^{\otimes m} \!.
\]
In short, apply $\tilde{\sigma}$ to factors corresponding to non-trivial components and do nothing to the other factors.  The inverse $\sigma^{-1}_{z,w}$ is given by the same recipe.  The sum of these maps over all resolutions of $\diagram$ is a pair of maps $\sigma_{z,w} \co \CSz_{z,w}(\diagram) \to \CSz(\diagram)$ and $\sigma_{z,w}^{-1}\co \CSz(\diagram) \to \CSz_{z,w}(\diagram)$.

\begin{Def}
  Let $\diagram$ be a knot diagram in the annulus $S^2 \setminus \{z,w\}$.  Fix decorations $\decos$.  Define $\partial_{z,w} \co \CSz_{z,w}(\diagram) \to \CSz_{z,w}(\diagram)$ by
  \[
    \partial_{z,w; \decos} = \sigma^{-1}_{z,w} \circ \partial_\decos \circ \sigma_{z,w}.
  \]
  The pair $(\CSz_{z,w}(\diagram), \partial_{z,w;\decos})$ is the \emph{doubly-pointed \Szabo chain complex} of $\diagram$.
\end{Def}

\begin{Prop}
  $(\CSz_{z,w}(\diagram), \partial_{z,w;\decos})$ is a chain complex.
\end{Prop}
\begin{proof}
This follows directly from the fact that $\partial_\decos$ is a differential.
\[
  \partial_{z,w;\decos}^2 
  = (\sigma^{-1}_{z,w}\,\partial_\decos\,\sigma_{z,w})^2 
  = \sigma^{-1}_{z,w}\,\partial_\decos \, \sigma_{z,w} \, \sigma^{-1}_{z,w} \, \partial_\decos \, \sigma_{z,w} 
  = \sigma^{-1}_{z,w} \, \partial_{\decos}^2 \, \sigma_{z,w} = 0.
\]
\end{proof}
We are not using anything particular to link homology: if $(V,d)$ is a complex and $\sigma \co W \to V$ is an invertible, linear map, then $(W, \sigma^{-1} d \sigma)$.  If not for the $k$ grading, then $\CSz_{z,w}$ would differ from $\CSz$ essentially by a change of basis.  Moreover, if $z$ and $w$ lie in the same region of $S^2 \setminus \diagram$ then $\CSz_{z,w}(\diagram) \cong \CSz(\diagram)$ as bigraded vector spaces and $k \equiv 0$.

A map like $\sigma$ appeared in \cite{Grigsby2015} under the name $\Theta$.  This map was used to study representation-theoretic interpretations of Lee's deformation of link homology and the annular grading.  We do not know if the present work is relevant, but it is interesting to see that \Szabo and Lee's theories have been connected in \cite{Sarkar2016} with conjectural applications to Heegaard Floer homology.  So it's promising to see this map appear here.

The Khovanov differential is $k$-filtered (i.e.\ non-increasing in $k$), but the full \Szabo differential is not.  (This is easy to arrange in Configuration 1, Figure \ref{fig:configuration1}.)  Neither is $\partial_{z,w;\decos}$ see Table~\ref{table:table}.  The last row of the table is tantalizing: it shows the only configuration whose map increases $k$.  It may be possible to change the map assigned to an $E$ configuration to remove this possibility.  Then it would be possible to define ``annular \Szabo homology'' in a much more straightforward way.

\begin{table}
	\noindent
	\begin{tabular}[h]{m{130pt}ccccc}
	 & \includegraphics[width=11pt]{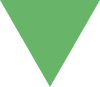} & \includegraphics[	width=11pt]{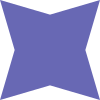} & \includegraphics[width=11pt]{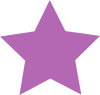} & \includegraphics[width=11pt]{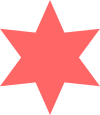} & \includegraphics[width=11pt]{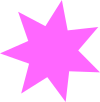} \\	
	A: \includegraphics[width=110pt]{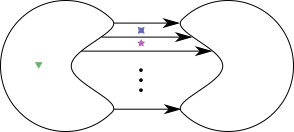} & \stack{$\Delta \tilde{q} = i - 2$}{$\Delta k = 0$} & \stack{$\Delta \tilde{q} = i - 4$}{$	\Delta k = -1$} &	 \stack{$	\Delta \tilde{q} = i - 6$}{$\Delta k = -2$} & & \\	
	B: \includegraphics[width=110pt]{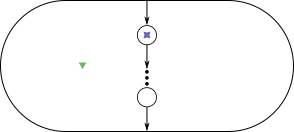} & \stack{$\Delta \tilde{q} = i - 2$}{$\Delta k = 0$} & \stack{$\Delta \tilde{q} 	= i - 6$}{$\Delta k = -2$} & & & \\	
	C: \includegraphics[width=110pt]{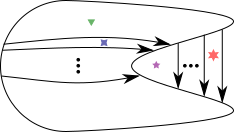} & \stack{$\Delta \tilde{q} = 	i - 2$}{$\Delta k = 0$} & \stack{$\Delta \tilde{q} 	= i - 2$}{$\Delta k = 0$} & \stack{	$\Delta \tilde{q} = i - 2$}{$\Delta k = 0$} & \stack{$\Delta \tilde{q} = i -6$}{$\Delta k = -2	$}	 & \\	
	D: \includegraphics[width=110pt]{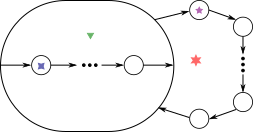} & \stack{$\Delta \tilde{q} = i - 2$	}{$\Delta k = 0$} & \stack{$\Delta \tilde{q} = i - 6$}{$\Delta k = -2$} & \stack{$	\Delta \tilde{q} = i - 2$}{$\Delta k = 0$}  &\stack{$\Delta \tilde{q} = i - 2$}{$\Delta k = 0$}	 	& \\	
	E: \includegraphics[width=110pt]{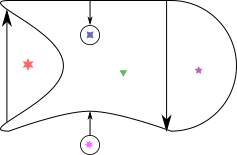} &\stack{$\Delta \tilde{q}  = i + 2$	}{$\Delta k = 2$} & \stack{$\Delta \tilde{q} = i - 2$}{$\Delta k = 0$} & \stack{$\Delta \tilde{q} = i - 2$}{$\Delta k = 0$} & \stack{$\Delta \tilde{q} = i - 2$}{$\Delta k = 0$}	  & \stack{$\Delta q = i - 2$}{$\Delta k = 0$}\\
	\end{tabular}
	\caption{In the leftmost column, the five families of configurations to which \Szabo assigns non-zero maps.  Each is $i$-dimensional.  The colored markings are basepoints.  The numbers shown are the $\tilde{q}$- and $k$-degrees of the maps $\cmap_{z,w;\config}$ where $w$ is the colored marking and $z$ is the point at infinity.}\label{table:table}
\end{table}

\begin{Prop}\label{prop:newGradingRule}
  Write $\partial_{z,w;i;\decos}$ for the part of $\partial_{z,w;\decos}$ which counts configurations of size $i$.  $\partial_{z,w;i;\decos}$ has $(q - 2k)$-degree $2i - 2$.  Equivalently, it has $(q - 2h - 2k)$-degree $-2$.
\end{Prop}

We can summarize this proposition as ``$\cmap_{z,w;\config}$ satisfies an \emph{annular grading rule}, i.e. the grading rule with $\tilde{q}$ replaced by $\tilde{q} - 2k$.''

\begin{proof}
Up to a global shift, $q = \tilde{q} + h$.  The map assigned to an $i$-configuration has $h$-degree $i$.  So it suffices to show that the $(\tilde{q} - 2k)$-degree of $\cmap_{z,w;\config}$ is $i - 2$.

Schematically, there are only finitely many places where one can place two basepoints on a configuration with non-zero map.  Table~\ref{table:table} shows that, in any case, the $(\tilde{q}-2k)$-degree of $\cmap_{z,w;\config}$ is $i - 2$.  (To be more precise: work by induction on the size of configurations.  The table establishes all base cases.  All larger configurations with non-zero maps are given by adding arrows and circles to those base cases.  It is easy to see that the number of non-trivial components does not change under these additions and therefore that $\Delta k$ does not change.  It also follows that $\Delta \tilde{q}$ is proportional to $i$ just as in the case without basepoints.)
\end{proof}

\begin{Rem}
It would be nice to have a more conceptual proof of Proposition~\ref{prop:newGradingRule}.  For example, let $w$ and $w'$ be basepoints separated by a single arc of $C_0$.  Suppose that $\cmap_{z,w;\config}$ satisfies the annular grading rule.  (This is always the case if $z$ and $w$ lie in the same region of $S^2 \setminus \diagram$.)  One could prove the theorem by showing that $\cmap_{z,w';\config}$ also satisfies the annular grading rule.  But it seems challenging to prove this by purely formal considerations.
\end{Rem}

\begin{Prop}\label{prop:filteredinvariant}
The $(h,q-2k)$-bigraded homology of $(\CSz_{z,w}(\diagram), \partial_{z,w})$ is an invariant of annular links.  In particular, it does not depend on a choice of decorations or diagrams.
\end{Prop}

We will need the following standard lemma.

\begin{Lemma}\label{lem:filteredisomorphism}
Let $f \co C \to C'$ be a filtered map of chain complexes with bounded $\bz$-filtrations.  Write $k'$ for the grading underlying the filtration.  Suppose that $f$ can be written as $f = f_0 + f_{> 0}$ so that $f_0$ is a filtration-preserving isomorphism and $f_{>0}$ is a sum of maps of lower-order, i.e. $k'(f_{>0}(x)) < k'(x)$ or $f_{>0}(x) = 0$ for all $x \in C$.  Then $f$ is an isomorphism of filtered complexes.
\end{Lemma}
\begin{proof}
Suppose that $f(x) = 0$ for some $x$.  Then $f_0(x) + f_{> 0}(x) = 0$, and in fact $f_0(x) = 0$; otherwise, $f_0(x) = - f_{>0}(x) \neq 0$, which is ruled out by the definition of each map.  Injectivity of $f_0$ implies that $x = 0$, so $f$ is injective.

Suppose that $y'$ is in the lowest filtration level of $C'$.  Then $f_{>0}(f_0^{-1}(y')) = 0$, so $f(f_0^{-1}(y')) = y'$.  So the lowest filtration level is in the image of $f$.  Write $C_i$ for all elements $x$ of $C$ with $k(x) \leq i$.  Let $w' \in C'_i$.  We work by induction on $i$: suppose that $C_{i-1} \subset \im(f)$.  By hypothesis, there is some $u \in C$ so that $f(u) = f_{>0}(f_0^{-1}(w'))$.  Therefore $f(f_0^{-1}(w') - u) = w' + f_{> 0}(f_0^{-1}(w')) - f_{>0}(f_0^{-1}(w'))$.  We conclude that $w'$ is in the image of $f$.  By induction, all of $C'$ is in the image.
\end{proof}

\begin{proof}[Proof of Proposition~\ref{prop:filteredinvariant}]
To show invariance under change of decoration, it suffices to consider decorations $\decos$ and $\decos'$ which differ at one crossing.  \Szabo defined so-called \emph{edge homotopy maps} $H \co \CSz(\diagram; \decos) \to \CSz(\diagram; \decos')$ and showed that $\Id(\decos, \decos') + H$ is a $q$-filtered isomorphism between chain complexes using Lemma~\ref{lem:filteredisomorphism}.

Define $H_{z,w} = \sigma_{z,w} H\sigma_{z,w}^{-1}$ and define $G = \Id + H_{z,w}$.  That $G$ is a chain map follows from the fact that $H_{z,w}$ satisfies the same formal relations as $H$, see section 6 of~\cite{Szabo2013}.   An argument like the proof of Proposition~\ref{prop:newGradingRule} shows that $H_{z,w}$ has $(q-2k)$-degree $1$.  Invariance under of change of decoration follows from Lemma~\ref{lem:filteredisomorphism}.

Khovanov showed that Khovanov homology is a link invariant by assigning a quasi-isomorphism to each Reidemeister move.  \Szabo showed that, with the right decorations, one can do the same for $\CSz$.  Having shown invariance under change of decoration, we can assume that we are using the right ones.

Suppose that $\diagram$ and $\diagram'$ are annular link diagrams which differ by a single annular Reidemeister move, i.e a Reidemeister move supported away from $z$ and $w$.  Write $\partial$ and $\partial'$ for the differentials on $\CSz_{z,w}(\diagram)$ and $\CSz_{z,w}(\diagram')$.  Write $\rho \co \CSz_{z,w}(\diagram) \to \CSz_{z,w}(\diagram')$ and $\rho' \co \CSz_{z,w}(\diagram') \to \CSz_{z,w}(\diagram)$ for the pair of chain maps defined by \Szabo and $J \co \CSz_{z,w}(\diagram) \to \CSz_{z,w}(\diagram)$ for the homotopy from $\rho\rho'$ to $\Id$.  Let $\rho_{z,w} \co \CSz_{z,w}(\diagram) \to \CSz_{z,w}(\diagram')$ be defined by $\rho_{z,w} = \sigma_{z,w}^{-1} \rho \sigma_{z,w}$.  Similarly, let $J_{z,w} = \sigma _{z,w}^{-1} J  \sigma_{z,w}$.  We have
\[
\partial'_{z,w}\rho_{z,w} + \rho_{z,w}\partial_{z,w} = \sigma_{z,w}\partial\rho\sigma_{z,w}^{-1} + \sigma_{z,w} \rho \partial \sigma_{z,w}^{-1} = \sigma_{z,w}(\partial\rho + \rho \partial) \sigma_{z,w}^{-1} = 0.
\]
An identical argument shows that $\rho'_{z,w}$ is a chain map.  Similarly,
\[
\rho_{z,w}\rho'_{z,w} = \sigma_{z,w}(\rho \rho') \sigma^{-1}_{z,w} = \sigma_{z,w}(J\partial' + \partial J)\sigma^{-1}_{z,w} = J_{z,w}\partial'_{z,w} + \partial_{z,w}J_{z,w}.
\]
\end{proof}

From now on we write $\Sz_{z,w}(\diagram)$ for the total homology of $\CSz_{z,w}(\diagram)$ and leave the decorations out of the notation.

In the un-pointed case, $\partial_1$ is exactly the Khovanov differential.  Write $\CKh_{z,w}(\diagram)$ for the complex $(\CSz_{z,w}(\diagram),\partial_{z,w;1})$.  We summarize the relationships between all these groups below.  Write $\CSz^{i,j}$ for the part of $\CSz$ with $(h,q)$-grading $(i,j)$.  Write $\CSz^{i,j,k}_{z,w}$ for the part of $\CSz$ with $(h,q,k)$-grading $(i,j,k)$.

\begin{Thm}\label{thm:2ptszabosummary} The map $\sigma_{z,w} \co \CSz(\diagram) \to \CSz_{z,w}(\diagram)$ is a chain map and induces a map $\sigma^*_{z,w} \co\Sz_{z,w}(\diagram) \cong \Sz(\diagram)$.  It enjoys the following properties:
\begin{enumerate}
	\item $\sigma_{z,w}$ is a $(q-2k)$-graded isomorphism, i.e.\ it carries $\bigoplus_{k \in \bz} \CSz_{z,w}^{i,j - 2k,k}(\diagram)$ to $\CSz^{i,j}(\diagram)$.
	\item $\sigma_{z,w}$ is $q$ graded if and only if $z$ and $w$ lie in the same region of $S^2 \setminus \diagram$.
	\item $\sigma_{z,w}^*$ induces a map $\sigma_{z,w;1}^* \co \Kh_{z,w}(\diagram) \cong \Kh(\diagram)$.  This map is $q$-graded if and only if $z$ and $w$ lie in the same region of $S^2 \setminus \diagram$.
	\item There is a spectral sequence from $\CSz_{z,w}(\diagram)$ converging to $\Sz_{z,w}(\diagram)$ whose second page is $\Kh_{z,w}(\diagram)$.  The map $\sigma_{z,w}$ induces an isomorphism of spectral sequences which agrees with $\sigma^*_{z,w;1}$ on the second page.
\end{enumerate}
\end{Thm}
\begin{proof}
The first point was proved in the proof of Proposition~\ref{prop:newGradingRule}.  The second is obvious.  All that's left for the third point is that $\sigma_{z,w}$ is a $\partial_{z,w;1}$-chain map, and this is clear as $\partial_{z,w;1}$ is exactly the part of $\partial_{z,w}$ which shifts the homological grading by $1$.  The existence of the spectral sequence in property (4) follows from the usual Leray construction.  As $\sigma_{z,w}$ does not shift homological grading, it induces a isomorphism of spectral sequences.
\end{proof}

Write $\ESz_{z,w}(\diagram)$ for the spectral sequence from item (4) of the theorem.

\subsection{Gradings from Knot Floer homology and a new conjecture}

In Proposition~\ref{prop:firstIsomorphism} we gave a graded correspondence between generators of $\CF(\branched(I))$ and generators of $\CSz(\diagram(I))$.  We saw at the end of Section~\ref{subsec:szabo} that this correspondence does not induce a chain map between $\branched(\diagram)$ and $\CSz(\diagram)$.  In fact, the differential on $\branched(\diagram)$ does not even satisfy the grading rule.  In this section, we propose that $\CF(\branched(\diagram(I))$ should instead be compared to $\CSz_{z,w}(\diagram(I))$.  In other words, $\BR(\diagram)$ satisfies an annular grading rule.  To formulate this conjecture we first add an annular-type grading to $\CF(\branched(\diagram(I))$.

\begin{Def}[\cite{Ozsvath2008}] A \emph{balanced, $2\ell$-pointed Heegaard diagram} consists of $(\Sigma, \bal, \bbe, \basew, \basez)$ where
\begin{itemize}
\item $(\Sigma, \bal, \bbe, \basew)$ is a \emph{balanced, $\ell$-pointed Heegaard diagram}.
\item $\basew = \{w_1, \ldots, w_\ell\}$ and $\basez = \{z_1, \ldots, z_\ell\}$  are each collections of $\ell$ basepoints so that, for each $i$, the points $z_i$ and $w_i$ lie in the same component of $\Sigma \setminus \bal$ and $\Sigma \setminus \bbe$.
\end{itemize}
The diagram $(\Sigma, \bal, \bbe, \basew)$ presents a three-manifold $Y$.  The complete diagram $(\Sigma, \bal, \bbe, \basew, \basez)$ presents a link $L \subset S^3$.  (Note that the case $\ell = 1$ is not the same as the doubly-pointed diagrams from Section \ref{sec:ozssz}!  $w$ and $z$ basepoints are handled differently.)
\end{Def}

Recall that every branched diagram has a pair of basepoints $w$ and $w'$, the pre-images of a basepoint on the link diagram $\diagram$.  Place a second basepoint on $\diagram$ and write $\basez = \{z, z'\}$ for its pre-images on the Heegaard surface $V$.  The diagram
\[
  \branched_{\basez,\basew}(I) = (V, \boldA, \boldB(I), \basew, \basez)
\]
is a balanced, $4$-pointed diagram for a link $\liftA$ in $S^3(\diagram(I))$.  The two basepoints on $S^2$ describe a bridge splitting of an oriented unknot.  The link $\liftA$ is the pre-image of this unknot in $S^3(\diagram(I))$.  For example, if $\diagram$ is a diagram for an annular braid closure in $S^2 \setminus \{z,w\}$, then $\liftA$ is the pre-image of the braid axis.  In this case, $\liftA$ has one component if $\diagram$ is the closure of a braid on an odd number of strands and two components otherwise.  In any case $\liftA$ is null-homologous.

Heegaard Floer homology was extended to links shortly after it's introduction, and since then the knot Floer homology package has grown to throng of powerful invariants, see~\cite{Manolescu2016} for a summary and further references.  In this paper we use only the additional filtration given by the knot Floer package, the \emph{Alexander filtration} $\alex$.  Let $x$ and $y$ be generators of $\CF(\branched(I))$ so that $\pi_2(x,y)$ is non-empty, and let $\phi \in \pi_2(x,y)$.  It turns out that the following definition is consistent, i.e. it does not depend on a choice of $\phi$ as long as $\pi_2(x,y)$ is non-empty.
\[
\alex(x) - \alex(y) = \left(n_{z_1}(\phi) + n_{z_2}(\phi)\right) - \left(n_{w_1}(\phi) + n_{w_2}(\phi)\right).
\]
As we have described it, $\alex$ is a relative filtration which induces a grading on standard generators.  $\CF(\branched(I))$ also has the absolute Maslov grading $\gr$ described in $\ref{subsec:minimality}$.\footnote{As in that section, we only consider torsion $\Spinc$ structures and therefore may consider absolute Maslov gradings.}  

\begin{Def}\label{def:cfkish}
  Write $\CF_{\basez}(\branched(I))$ for the chain complex $\CF(\branched(I))$ with the filtration given by $\alex$ for the knot $\liftA$ described by $\basew$ and $\basez$.
\end{Def}

\begin{Prop}\label{prop:finalgradings}
Equip $\CF_{\basez}(\branched(I))$ with bigrading $(2\gr, 2\alex)$. Equip $\CSz_{z,w}(\diagram(I))$ with the bigrading $(q, k)$.  The isomorphism of vector spaces of Proposition~\ref{prop:firstIsomorphism} defines a bigraded isomorphism $\CF_{\basez}(\branched(I), \liftA) \cong \CSz_{z,w}(\diagram(I))$.
\end{Prop}

\begin{proof}
  The only thing to check is that that $2k$ agrees with $\alex$.  Let $x$ and $y$ be generators of $\CF_{\basez}(\branched(I))$ which agree on all but one component of $\branched(I)$.  Call that component $Q$.  As in the proof of Proposition~\ref{prop:firstIsomorphism}, we can work component by component.  There is a pair of domains representing Whitney disks $\phi_1, \phi_2 \in \pi_1(x,y)$ which contribute to the differential on $\CF(\branched(I))$.  If $Q$ is non-trivial, then $n_{z_1}(\phi_1) = 1$, $n_{z_2}(\phi_1) = 0$, $n_{z_1}(\phi_2) = 0$, and $n_{z_2}(\phi_2) = 1$.   So $\alex(x) - \alex(y) = 1$, while $k(x) - k(y) = 2$.

  If $Q$ is trivial, then the same argument shows that $\alex(x) - \alex(y) = 0$ and of course $k(x) - k(y) = 0$.
\end{proof}

\begin{Rem}\label{rem:grading}
It is interesting to compare this result with \cite{Grigsby2011} where the Alexander grading is connected with the annular grading via sutured Floer homology.  Grigsby and Wehrli also give a representation-theoretic interpretation.
\end{Rem}

We developed $\CSz_{z,w}$ to correct for the discrepancy of Section~\ref{subsec:szabo}.  Indeed, $\CSz_{z,w}$ agrees with $\BR(\diagram)$ for two-dimensional configurations in the following sense.  Let $\config_{z,w} \co C_0 \to C_1$ be a doubly-pointed configuration of size two.  It defines a map $\cmap_{z,w;\config} \co \CSz_{z,w}(C_0) \to \CSz_{z,w}(C_1)$.  There is also  a map $f_{\basez,\basew} \co \CF_{\basez}(\branched(C_0)) \to \CF_{\basez}(\branched(C_1))$ described in Section \ref{sec:ozssz} (for bouquet diagrams) which counts holomorphic rectangles in a multi-diagram which interpolates between $\branched(C_0)$ and $\branched(C_1)$.  These rectangles should not intersect $\basew$, and their contribution to $\alex$ is determined by their intersection number with $\basez$.

\begin{Prop}\label{prop:agreement}
Let $I_0 \co \CSz_{z,w}(C_0) \to \CF_{\basez}(\branched(C_0))$ and $I_1 \co \CSz_{z,w}(C_1) \to \CF_{\basez}(\branched(C_1))$ be the isomorphism of Proposition~\ref{prop:finalgradings}.  For some set of decorations on $\config$, the following diagram commutes.
\[
\begin{tikzcd}
  \CSz_{z,w}(C_0) \arrow[r,"\cmap_{z,w;\config}"] \arrow[d,"I_0"]
  & \CSz_{z,w}(C_1) \arrow[d,"I_0"] \\
  \CF_{\basez}(\branched(C_0)) \arrow[r,"f_{\basez,\basew}"] 
  & \CF_{\basez}(\branched(C_1))
\end{tikzcd}
\]
\end{Prop}

\begin{proof}
There are sixteen two-dimensional configurations described in~\cite{Szabo2013}, or eight after ignoring decorations.  To prove the proposition, one must draw the eight Heegaard multi-diagrams as in Figure~\ref{fig:notriangle} and identify the unique domain representing a holomorphic rectangle between the appropriate generators.  A periodic domain argument shows that there are no other holomorphic rectangles.
\end{proof}

Proposition~\ref{prop:agreement} may be seen as an enrichment of Theorem 6.3 of \cite{Ozsvath2005b}.  There it is shown that there is an isomorphism between $\OS^1(\diagram)$ and $\CKh(\diagram)$ so that the differentials agree.  Proposition~\ref{prop:agreement} shows that the isomorphism of vector spaces in Proposition \ref{prop:finalgradings} intertwines the differentials on $\BR^2(\diagram)$ and $\ESz^2_{z,w}(\diagram)$.

\subsection{Further constructions and speculation}

Branched diagrams were introduced in a different context by Grigsby and Wehrli.  They use them to prove the following theorem which is closely related to a theorem of Roberts \cite{Roberts2013}.

\begin{Thm*}\cite{Grigsby2010}
There is a spectral sequence from $\AKh(L)$ to a $\widetilde{\mathrm{HFK}}(S^3(L), \liftA)$ where $\liftA$ is the annular axis's pre-image in $S^3(L)$. 
\end{Thm*}

Here $\widetilde{\mathrm{HFK}}$ is a minor variation of the knot Floer homology $\HFK$ and $\AKh(L)$ is the \emph{annular Khovanov homology of $L$}.  This is the homology of $\CKh(L)$ in which uses only the part of the differential which preserves the $k$-grading.  Grigsby and Wehrli's proof uses \emph{sutured Floer homology}, a version of Heegaard Floer homology which allows for nice cut-and-paste arguments.\footnote{This why $\AKh$ has also been called sutured Khovanov homology.}  Their argument adapts \Ozsvath and \Szabo's spectral sequence to the sutured world and identifies the target with knot Floer homology.  These results and Remark \ref{rem:grading} inspire the following conjecture.

\begin{Conj}\label{conj:me1}
Let $\diagram$ be a spherical link diagram with two basepoints, $w$ and $z$.  Let $A$ be an unknot which meets the sphere at $w$ and $z$.  Let $Y = \bigoplus_{I \in \{0,1\}} \CF_{\basez}(\branched(I))$ and define the differential as in Sections \ref{sec:ozssz} and \ref{sec:handleslides}.

The order filtration on $Y$ induces a spectral sequence from $\Kh_{z,w}(m(L))$ to $\HF(S^3(-L))$.  There is a grading $\delta_i$ on each page of the spectral sequence so that $\delta_0$ and $\delta_1$ agree with the Khovanov grading $q - - 2k - 2h$ and $\delta_\infty = 2(\gr - 2\alex) - 2h$ on $\HF(S^3(-L))$ with Alexander filtration given by $\liftA$.  With respect to this grading, the differential on the $i$th page has degree $-2$.

This spectral sequence is isomorphic to $\OS(\diagram)$, and the isomorphism (and its grading behavior) depends on $A$.  The spectral sequence is also isomorphic to $\ESz_{z,w}(\diagram)$ by the isomorphism of Proposition \ref{prop:finalgradings}.
\end{Conj}

This conjecture may seem more complex than Seed and \Szabo's; after all, we are adding in an unknot and new filtration.  But it suggests that the isomorphism between $\OS(\diagram)$ (or $\BR(\diagram)$) and $\ESz(\diagram)$  will depend on additional geometric data: a choice of an unknot or an \emph{annular} embedding of $\diagram$.  It may be difficult to prove the original conjecture without taking that data into account.

The combination $\alex - \gr$ is well-known in the knot Floer world as the $\delta$-grading.  Various combinations of $q$ and $h$ in Khovanov homology are also called $\delta$-gradings.  The two are compared by Baldwin, Levine, and Sarkar in \cite{Baldwin2015} (although their conventions differ from those of \cite{Szabo2013}).  The grading $q - 2k$ plays a role in \cite{Grigsby2016}.

More powerful knot Floer invariants come from an algebraically richer theory, $\text{CFK}^-$ which is defined over a polynomial ring $\bf[U]$.  One can recover $\CF$ from $\text{CFK}^-$ by setting $U$ to one: this transforms a graded theory over a polynomial ring into a filtered theory.  By reversing this recipe, one can construct a doubly-pointed \Szabo homology over a polynomial ring $\bf[W]$.  It is fun to speculate that $\CSz^-$ might be a combinatorial model for some version of $\text{CFK}^-$.  See~\cite{Sarkar2016} for more on these recipes and polynomial actions in Khovanov homology.

\bibliography{specseqsh}

\end{document}